\documentclass{article}
\usepackage{graphicx}
\usepackage{amssymb}
\usepackage{amsmath}
\usepackage{amsthm}
\usepackage{enumerate}
\usepackage{color}
\newtheorem{theorem}{Theorem}[section]
\newtheorem{lemma}[theorem]{Lemma}
\newtheorem{corollary}[theorem]{Corollary}
\newtheorem{proposition}[theorem]{Proposition}
\theoremstyle{definition}
\newtheorem{remark}[theorem]{Remark}

\newtheorem{definition}[theorem]{Definition}
\numberwithin{equation}{section}

\allowdisplaybreaks[1]

\newcommand{\open}[1]{\smallskip\noindent\fbox{\parbox{\textwidth}{\color{blue}\bfseries\begin{center}
      #1 \end{center}}}\\ \smallskip}

\newcommand{\sdm}[1]{{\color{blue} #1}}
\renewcommand{\sdm}[1]{{#1}}
\newcommand{\ignore}[1]{}
\newcommand{\bli}{\begin{list}{}{\labelwidth6mm\leftmargin8mm}}
\newcommand{\eli}{\end{list}}

\newcommand{\dint}{\;\mathrm{d}}
\newcommand{\Dd}{\;\mathrm{D}}
\newcommand{\whole}[1]{\ensuremath\left\lfloor #1 \right\rfloor}

\newcommand{\fz}{\infty}
\def\ls{\lesssim}
\def\vz{\varphi}
\def\supp{{\mathop\mathrm{supp\,}\nolimits}}

\newcommand{\bt}{{B}_{p,q}^{s,\tau}}

\newcommand{\ft}{{F}_{p,q}^{s,\tau}}

\newcommand{\MA}{\ensuremath{{\cal A}^{s}_{u,p,q}}}

\newcommand{\MB}{\ensuremath{{\cal N}^{s}_{u,p,q}}}

\newcommand{\MF}{\ensuremath{{\cal E}^{s}_{u,p,q}}}

\newcommand{\M}{\ensuremath{{\cal M}_{u,p}}}
\newcommand{\cM}{\ensuremath{{\mathcal M}}}
\newcommand{\SRd}{\mathcal{S}(\rd)}
\newcommand{\SpRd}{\mathcal{S}'(\rd)}
\newcommand{\bmo}{\mathop{\mathrm{bmo}}}
\newcommand{\Lloc}{L_1^{\mathrm{loc}}}

\newcommand{\A}{\ensuremath{A^s_{p,q}}}

\newcommand{\beq}{\begin{equation}}
\newcommand{\eeq}{\end{equation}}
\newcommand{\ve}{\varepsilon}

\newcommand{\real}{\ensuremath{{\mathbb R}}}
\newcommand{\rr}{{\real}}
\newcommand{\rd}{\ensuremath{{\real^d}}}
\newcommand{\cc}{\ensuremath{{\mathbb C}}}
\newcommand{\zz}{\ensuremath{{\mathbb Z}}}
\newcommand{\zd}{\ensuremath{{\zz^d}}}
\newcommand{\nn}{\ensuremath{{\mathbb N}}}
\newcommand{\no}{\ensuremath{{\mathbb N}_0}}
\newcommand{\cs}{\ensuremath{\mathcal S}}

\usepackage{cite}
\newcommand{\MBd}{\ensuremath{{\cal N}^{\frac{d}{u}}_{u,p,q}}}
\newcommand{\MFd}{\ensuremath{{\cal E}^{\frac{d}{u}}_{u,p,q}}}

\newcommand{\cN}{\ensuremath{{\cal N}}}
\newcommand{\cE}{\ensuremath{{\cal E}}}

\begin{document}

\title{Some embeddings of Morrey spaces with critical smoothness 
}

\author{Dorothee D. Haroske\footnotemark[1],
Susana D. Moura\footnotemark[3],   and Leszek Skrzypczak\footnotemark[1] \footnotemark[2]}
\maketitle

\footnotetext[1]{The first and third author were partially supported by the German Research Foundation (DFG), Grant no. Ha 2794/8-1.}
\footnotetext[2]{The author was partially supported by National Science Center, Poland,  Grant No.
  2013/10/A/ST1/00091.}
 \footnotetext[3]{The author was partially supported by the Centre for Mathematics of the University of Coimbra -- UID/MAT/00324/2019, funded by the Portuguese Government through FCT/MEC and co-funded by the European Regional Development Fund through the Partnership Agreement PT2020}


\begin{abstract} 
\sdm{We study embeddings of Besov-Morrey spaces  $\MB(\rd)$ and of Triebel-Lizorkin-Morrey spaces $\MF(\rd)$  in the limiting cases when the smoothness  $s$ equals  $s_0=d\max(1/u-p/u,0)$ or $s_{\infty}=d/u$, which is related to the embeddings in $\Lloc(\rd)$ or in $L_{\infty}(\rd)$, respectively.  When $s=s_0$ we characterise the embeddings in $\Lloc(\rd)$ and when  $s=s_{\infty}$ we obtain embeddings into Orlicz-Morrey spaces of exponential type and into generalised Morrey spaces. }   
\end{abstract}

\medskip

 \smallskip


\section{Introduction} 

In recent years  function spaces built upon Morrey spaces $\M(\rd)$, $0 < p \le u < \infty$, attracted some attention.
They  include in particular Besov-Morrey spaces $\MB(\rd)$ and  Triebel-Lizorkin-Morrey spaces $\MF(\rd)$, $0 < p \le u < \infty$, $0 < q \le \infty$, $s \in \real$. 
The attention paid to the spaces was motivated first of all by possible applications.  
The Besov-Morrey spaces $\MB(\rd)$ were introduced by Kozono and Yamazaki in \cite{KY} and used by them and later on by Mazzucato \cite{Maz} in the study of Navier-Stokes equations. In \cite{TX} Tang and Xu introduced the corresponding Triebel-Lizorkin-Morrey spaces $\MF(\rd)$, thanks to establishing the Morrey version of the Fefferman-Stein vector-valued inequality. Some properties of these spaces including their atomic decompositions and wavelet characterisations were later described in the papers by Sawano \cite{Saw2,Saw1}, Sawano and Tanaka \cite{ST2,ST1} and Rosenthal \cite{MR-1}. 
Also embedding properties of these spaces were investigated in a series of papers \cite{hs12,hs13}.

In our two recent papers \cite{HaSM3} and  \cite{hms} we  studied  under which conditions  the spaces 
 $\MB(\rd)$ and  $\MF(\rd)$ contain only regular distributions, i.e.,  when
\begin{equation}\label{int0}
\MB(\rd)\subset\Lloc(\rd), \quad \MF(\rd)\subset\Lloc(\rd)
\end{equation}
and when they  consist  of bounded functions
\begin{equation}\label{int1}
\MB(\rd)\subset L_\infty(\rd), \quad \MF(\rd)\subset L_\infty(\rd).
\end{equation}
The  embeddings \eqref{int0} hold if the smoothness $s$ is  related to  $s_o=\frac{p}{u}d \max(\frac{1}{p}-1,0)$. Similarly the embeddings \eqref{int1} are valid if $s$ is  related to  $s_\infty=\frac{d}{u}$. Therefore we called $s_o$ and $s_\infty$ critical smoothnesses. In particular,    the spaces contain a non-regular  distribution if $s<s_o$ and  consist of locally integrable functions if  $s>s_o$. The behaviour in the critical case $s=s_o$  is a delicate question. The analogous situation occurs in the case of boundedness of the functions.

%
%

This paper is the direct continuation of the above mentioned recent  articles. First in Section 3 we prove that the embeddings \eqref{int0} with smoothness $s=s_o$ hold if and only if the spaces $\MB(\rd)$ and $\MF(\rd)$ are embedded into some Morrey spaces $\mathcal{M}_{v,r}(\rd)$ with properly defined indices  $v$ and $r$, cf. Theorem \ref{theorem-1-e} and  Theorem \ref{theorem-n_290319}. 

In Section 4 we  investigate unboundedness properties of the functions belonging to  smoothness spaces of Morrey type with $s=s_\infty$. First we prove that the spaces $\mathcal{N}^{d/u}_{u,p,q}(\rd)$   and $\mathcal{E}^{d/u}_{u,p,q}(\rd)$ can be embedded into Orlicz-Morrey spaces of exponential type, cf. Theorem 4.5 and Corollary 4.6. The idea that  exponential Young functions can control  the unboundedness of functions belonging to the Sobolev spaces of critical smoothness goes back to Trudinger \cite{tru}, cf. also Strichartz \cite{Str}. There are several definitions of Orlicz-Morrey spaces. Here we follow the approach proposed by Nakai \cite{nakai}. Afterwards we embed the above spaces of smoothness $\frac{d}{u}$ into the so-called generalised Morrey spaces, cf. Theorem \ref{theorem-MB} and Remark \ref{rmk-MB}. The generalised Morrey spaces have been investigated by several authors.  Our  results here are    close to that ones in papers by Sawano and Wadade \cite{sw}, Nakamura, Noi and Sawano \cite{nns} and  by Eridani, Gunawan, Nakai  and Sawano \cite{EGNS14}.  
However in contrast to the above papers we study also the Besov-Morrey spaces. Moreover in the case of Sobolev-Morrey spaces we are able to prove the embeddings for a larger set of parameters.   

In the very short last section, somewhat supplementary,   we formulate  some outcome concerning  embeddings into spaces with smoothness $0$, in particular to $\bmo (\rd)$. In Section 2 we collect the definition and basic facts concerning the smoothness Morrey spaces that are needed in the next sections.

\section{Smoothness spaces of Morrey type} \label{sect-2}

First we fix some notation. By $\nn$ we denote the \emph{set of natural numbers},
by $\nn_0$ the set $\nn \cup \{0\}$,  and by $\zz^d$ the \emph{set of all lattice points
in $\rd$ having integer components}. Let $\nn_0^d$, where $d\in\nn$, be the set of all multi-indices, $\alpha = (\alpha_1, \ldots,\alpha_d)$ with $\alpha_j\in\nn_0$ and $|\alpha| := \sum_{j=1}^d \alpha_j$. If $x=(x_1,\ldots,x_d)\in\rd$ and $\alpha = (\alpha_1, \ldots,\alpha_d)\in\nn_0^d$, then we put $x^\alpha := x_1^{\alpha_1} \cdots x_d^{\alpha_d}$. 
For $a\in\real$, let   $\whole{a}:=\max\{k\in\zz: k\leq a\}$ and $a_+:=\max(a,0)$.
If $0<u\leq \infty$, the number $u'$ is given by $\frac{1}{u'}=(1-\frac1u)_+$, with the convention that $1/\infty=0$. 
All unimportant positive constants will be denoted by $C$, 
occasionally the same letter $C$ is used to denote different constants  in the same chain of inequalities.
 By the notation $A \ls B$, we mean that there exists a positive constant $c$ such that
 $A \le c \,B$, whereas  the symbol $A \sim B$ stands for $A \ls B \ls A$.
We denote by $B(x,r) :=  \{y\in \rd: |x-y|<r\}$ the ball centred at $x\in\rd$ with radius $r>0$ and $|\cdot|$ denotes the Lebesgue measure when applied to measurable subsets of $\rd$.
Let $\mathcal{Q}$ be the collection of all \emph{dyadic cubes} in $\rd$, namely,
$\mathcal{Q}:= \{Q_{j,k}:= 2^{-j}([0,1)^d+k):\ j\in\zz,\ k\in\zz^d\}.$

Given two (quasi-)Banach spaces $X$ and $Y$, we write $X\hookrightarrow Y$
if $X\subset Y$ and the natural embedding of $X$ into $Y$ is continuous.

\medskip

We introduce smoothness  spaces of Morrey type. We recall first  the definition of Morrey spaces.

\begin{definition}
Let $0 <p\leq  u<\infty$. The \emph{Morrey space}
  $\M(\rd)$  is  the set of all
  locally $p$-integrable functions $f\in L_p^{\mathrm{loc}}(\rd)$  such that
\begin{equation} \label{Morrey-norm}
\|f \mid {\M(\rd)}\| :=\, \sup_{Q\in\mathcal{Q}} |Q|^{\frac{1}{u}-\frac{1}{p}}
\biggl(\int_{Q} |f(y)|^p \dint y \biggr)^{1/p}\, <\, \infty\, .
\end{equation}
\end{definition}

\begin{remark}
The spaces $\M(\rd)$ are quasi-Banach spaces (Banach spaces for $p \ge 1$){\color{blue}}, that can be equivalently defined if the supremum in \eqref{Morrey-norm} is taken over all balls $B(x,r)$ or over all 
cubes $Q(x,r)$, with center $x\in\rd$ and side length $r>0$, and sides parallel to the axis of coordinates. \\
These spaces originated from Morrey's study on PDE (see \cite{Mor}) and are part of the wider class of Morrey-Campanato spaces; cf. \cite{Pee}. They can be considered as a complement to $L_p$ spaces, since $\cM_{p,p}(\rd) = L_p(\rd)$ with $p\in(0,\infty)$, extended by $\cM_{\infty,\infty}(\rd)  = L_\infty(\rd)$. In a parallel way one can define the spaces $\cM_{\infty,p}(\rd)$, $p\in(0, \infty)$, but using the Lebesgue differentiation theorem, one arrives at $\cM_{\infty, p}(\rd) = L_\infty(\rd)$. Moreover, $\M(\rd)=\{0\}$ for $u<p$, and for  $0<p_2 \le p_1 \le u < \infty$,
\begin{equation*} 
	L_u(\rd)= \cM_{u,u}(\rd) \hookrightarrow  \cM_{u,p_1}(\rd)\hookrightarrow  \cM_{u,p_2}(\rd).
\end{equation*}
\end{remark}

\medskip
Let $\SRd$ be the set of all \emph{Schwartz functions} on $\rd$, endowed
with the usual topology,
and denote by $\SpRd$ its \emph{topological dual}, namely,
the space of all bounded linear functionals on $\SRd$
endowed with the weak $\ast$-topology. 
Let $\vz_0=\vz\in\SRd$ be such that
\begin{equation}\label{e1}
\supp{\vz}\subset \{x\in\rd:\,|x|\le2\}\, \qquad\text{and}\qquad
{\vz}(x)=1 \quad  \text{if}\quad  |x|\leq 1,
\end{equation}
and for each $j\in \nn$ let  $\vz_j(x):=\vz(2^{-j} x)-\varphi(2^{-j+1} x)$. Then $\{\varphi_j\}_{j=0}^{\infty}$ forms a smooth dyadic resolution of unity.


\begin{definition}
Let $0 <p\leq  u<\infty$ or $p=u=\infty$. Let  $q\in(0,\infty]$, $s\in \real$ and $\{\varphi_j\}_{j=0}^{\infty}$  a smooth dyadic resolution of unity.
\bli
\item[{\upshape\bfseries (i)}]
The  {\em Besov-Morrey   space}
  $\MB(\rd)$ is defined to be the set of all distributions $f\in \SpRd$ such that
\begin{align*}
\big\|f\mid \MB(\rd)\big\|:=
\bigg(\sum_{j=0}^{\infty}2^{jsq}\big\|{\mathcal F}^{-1} (\varphi_j {\mathcal F}f)\vert 
\M(\rd)\big\|^q \bigg)^{1/q} < \infty
\end{align*}
with the usual modification made in case of $q=\fz$.
\item[{\upshape\bfseries  (ii)}]
Let $u\in(0,\fz)$. The  {\em Triebel-Lizorkin-Morrey  space} $\MF(\rd)$
is defined to be the set of all distributions $f\in   \SpRd$ such that
\begin{align*}
\big\|f \mid \MF(\rd)\big\|:=\bigg\|\bigg(\sum_{j=0}^{\infty}2^{jsq} |
 {\mathcal F}^{-1} (\varphi_j {\mathcal F}f)(\cdot)|^q\bigg)^{1/q}
\mid \M(\rd)\bigg\| <\infty
\end{align*}
with the usual modification made in case of  $q=\fz$.
\eli
\end{definition}

\begin{remark}  Occasionally we adopt the usual custom  to  write $\MA(\rd)$ instead of   $\MB(\rd)$  or  $\MF(\rd)$, when both scales are meant simultaneously in some context. 
The  spaces $\MA(\rd)$  are independent of the particular dyadic partition of unity $\{\vz_j\}_{j=0}^{\infty}$ appearing in their definitions.
They are quasi-Banach spaces (Banach spaces for $p,\,q\geq 1$), and $\mathcal{S}(\rd) \hookrightarrow
\MA(\rd) \hookrightarrow \mathcal{S}'(\rd)$.  Moreover, for $u=p$  we re-obtain the usual  Besov and Triebel-Lizorkin spaces ${\cal A}^{s}_{p,p,q}(\rd) = \A(\rd)$. 
Besov-Morrey spaces were introduced by Kozono and Yamazaki in
\cite{KY}. They studied semi-linear heat equations and Navier-Stokes
equations with initial data belonging to  Besov-Morrey spaces.  The
investigations were continued by Mazzucato \cite{Maz}, where one can find the
atomic decomposition of some spaces.
The Triebel-Lizorkin-Morrey spaces
were later introduced by  Tang and Xu \cite{TX}, we follow their approach. 
The ideas were further developed by Sawano and Tanaka \cite{ST1,ST2,Saw1,Saw2}. Closely related, alternative approaches can  be found in the monographs \cite{ysy,t13,t14} or in the survey papers by Sickel \cite{s011,s011a}.
\end{remark}

We list some  elementary embeddings within this scale of spaces.  It holds
\begin{equation*} 
{\mathcal A}^{s+\varepsilon}_{u,p,r}(\rd)  \hookrightarrow
\MA(\rd), \quad\text{if} \quad \ve>0, \quad r,q\in(0,\infty],
\end{equation*}
and 
\begin{equation} \label{elem-1-t}
{\cal A}^{s}_{u,p,q_1}(\rd)  \hookrightarrow {\cal A}^{s}_{u,p,q_2}(\rd),\quad  
\quad\text{if} \quad q_1\le q_2.
\end{equation}
Sawano proved in \cite{Saw2} that, for $s\in\real$ and $0<p< u<\infty$,
\begin{equation}\label{elem}
	{\cal N}^s_{u,p,\min(p,q)}(\rd)\, \hookrightarrow \, \MF(\rd)\, \hookrightarrow \,{\cal N}^s_{u,p,\infty}(\rd),
\end{equation}
where, for the latter embedding, $r=\infty$ cannot be improved -- unlike
in case of $u=p$. More precisely,
$\MF(\rd)\hookrightarrow {\mathcal N}^s_{u,p,r}(\rd)$ {if, and only if,} $r=\infty$ {or} $u=p$ {and} $r\ge \max(p,q)$.
On the other hand, Mazzucato has shown in \cite[Prop.~4.1]{Maz} that 
\begin{equation}\label{mazzu}
{\mathcal E}^0_{u,p,2}(\rd)=\M(\rd), \quad 1<p\leq u<\infty,
\end{equation}
in particular,
$$
{\mathcal E}^0_{p,p,2}(\rd)=L_p(\rd), \quad 1<p<\infty.
$$
This is nothing else than the well-known classical coincidence $F^0_{p,2}(\rd)=L_p(\rd)$, $1<p<\infty$, cf. \cite[Thm.~2.5.6]{T-F1}.
Further embedding results for the above scales of function spaces on $\rd$ can be found in \cite{hs12,hs13,YHMSY,hms}.

\subsubsection*{The  atomic decompositions } 

An important tool in our later considerations is the
characterisation of the  Besov-Morrey and Triebel-Lizorkin-Morrey spaces  by means
of atomic decompositions. We follow \cite{MR-1}; see also \cite{ST2}. 

\begin{definition}\label{defatoms}  
Let $0<p\leq u< \infty$, $q\in(0,\infty]$ and $s\in \rr$. Let $K\in \no$ and $N\in \{-1\}\cup \no$. 
A collection of  $L_{\infty}$-functions $a_{jm}:\rd\to \cc$,  $j\in\no$, $m\in\zz^d$,  is a family of  $(K,N)$-atoms 
if, for some $c_1>1$ and $c_2>0$, it holds
\bli
\item[{\upshape\bfseries (i)}]
 $\supp a_{jm} \subset c_1 Q_{jm}, \quad j\in\no,\; m\in\zz^d$,
\item[{\upshape\bfseries (ii)}] there exist all (classical derivatives) $\Dd^{\alpha}a_{jm}$ for $\alpha\in\no^n$ with $ |\alpha|\leq K$ 
and
\begin{equation} \label{cond-atoms}
\|\Dd^{\alpha}a_{jm} | L_{\infty}(\rd)\|\leq c_2 2^{j |\alpha|}, \quad j\in\no,\; m\in\zz^d,
\end{equation}
\item[{\upshape\bfseries (iii)}] if $\gamma\in\nn_0^d$ with $|\gamma|\leq N$, then 
\begin{equation*} 
\int_{\rd}x^{\gamma}a_{jm}(x) \dint x=0, \quad  j\in\no,\;\; m\in\zz^d.
\end{equation*}
$N=-1$ means that no moment condition is required.  
\eli
\end{definition}

\sdm{For $0<u<\infty$, $j\in\no$ and $m\in\zz^d$ we denote by $\chi_{j,m}$ the characteristic function of the cube $Q_{j,m}$ and by  
$\chi_{j,m}^{(u)}:=2^{jd/u}\chi_{j,m}$ the $u$-normalised  characteristic function of the same cube, i.e.,  such that $\|\chi_{j,m}^{(u)}\mid\M(\rd) \|=1$.}

\begin{definition}\label{dts}
Let $0<p\leq u<\infty$, $q\in(0,\fz]$ and  $s\in \rr$.
\bli
\item[{\upshape\bfseries (i)}]
The \emph{sequence space}
$e^s_{u,p,q}(\rd)$ is defined to be the set
of all sequences $\lambda:=\{\lambda_{j,m}\}_{j\in\no,m\in\zz^d}\subset \cc$ such that
$$\|\lambda \mid e^s_{u,p,q}(\rd)\|:=
\Big\| \big(\sum_{j=0}^\infty  2^{jq(s-\frac{d}{u})} 
\sum_{m\in\zz^d} |\lambda_{j,m}|^q \chi^{(u)q}_{j, m}\big)^{1/q}|
\M(\rd)\Big\| <\infty$$
with the usual modification in case of $q=\fz$.
\item[{\upshape\bfseries (ii)}]
The \emph{sequence space}
$n^s_{u,p,q}(\rd)$ is defined to be the set
of all sequences $\lambda:=\{\lambda_{j,m}\}_{j\in\no,m\in\zz^d}\subset \cc$ such that
$$\|\lambda \mid n^s_{u,p,q}(\rd)\|:=
 \Big(\sum_{j=0}^\infty  2^{jq(s-\frac{d}{u})} \big\|
\sum_{m\in\zz^d} |\lambda_{j,m}|\chi^{(u)}_{j, m}\,|\M(\rd)\big\|^q\Big)^{1/q} <\infty$$
with the usual modification in case of $q=\fz$.
\eli
\end{definition}

\begin{remark} \label{rmk-sequencenorm}
It was proved in \cite{hs12} that 
\begin{align*}
\|\lambda \mid n^s_{u,p,q} & (\rd)\| \sim \|\lambda| {  n}^{s}_{u,p,q}\|^\ast : = \\
= &\Bigg(\sum_{j=0}^\infty 2^{qj(s-\frac d u)}\!
\sup_{\nu: \nu \le j; k\in \mathbb{Z}^d}\! 2^{qd(j-\nu)(\frac 1 u - \frac 1 p )}\Big(\sum_{m:Q_{j,m}\subset Q_{\nu,k}}\!\!|\lambda_{j,m}|^p\Big)^{\frac q p}\Bigg)^{1/q} < \infty . \nonumber
\end{align*}
\end{remark}

\smallskip

For $p,q\in(0, \infty]$, let 
\begin{equation}\label{sigma_p}
\sigma_p := d\left(\frac{1}{p} - 1\right)_{+} \quad \mbox{and} \quad  \sigma_{p,q} := d\left(\frac{1}{\min(p,q) }- 1\right)_{+} .
\end{equation}
According to \cite[Thms. 2.30 and 2.36]{MR-1} (see also \cite[Thm.~4.12]{ST2}), we have the  following 
atomic decomposition characterisation of $\MA(\rd)$, where we adopt the same custom to write $a^s_{u,p,q}(\rd)$ instead of $n^s_{u,p,q}(\rd)$ or $e^s_{u,p,q}(\rd)$, for convenience.

\begin{proposition}\label{atomicdecomposition}
Let $0<p\leq u<\infty$, $q\in(0,\fz]$ and  $s\in \rr$. Let

 $$
    K \ge \max(\whole{s+1},0)\\
$$
and  
$$
 N \ge \max(\whole{\sigma_{p,q}-s},-1) \; ({\cal E}- case)\quad \text{or}\quad   N\ge \max(\whole{\sigma_{p}-s},-1) \; ({\cal N}- case).  
$$

Then for each $f\in \MA(\rd)$, there exist a family $\{a_{jm}\}_{j\in\no,m\in\zz^d}$ of $(K,N)$- atoms 
and  a sequence  $\lambda=\{\lambda_{jm}\}_{j\in\no,m\in\zz^d}\in a^s_{u,p,q}(\rd)$ such that 
$$
f=\sum_{j=0}^{\infty}\sum_{m\in\zz^d}\lambda_{jm} \,a_{jm} \quad \mbox{in} \quad \cs'(\rd)
$$
and 
$$
\|\lambda \mid a^s_{u,p,q}(\rd)\|\leq C\, \|f \mid  \MA(\rd)\|,
$$
where $C$ is a positive constant independent of $\lambda$ and $f$.

Conversely, there exists a positive constant $C$ such that for all families $\{a_{jm}\}_{j\in\no,m\in\zz^d}$ of $(K,N)$-atoms 
 and $\lambda=\{\lambda_{jm}\}_{j\in\no,m\in\zz^d}\in a^s_{u,p,q}(\rd)$,
$$
 \Big\|\sum_{j=0}^{\infty}\sum_{m\in\zz^d}\lambda_{jm} \,a_{jm} \mid  \MA(\rd)\Big\|\leq C\,  \|\lambda \mid a^s_{u,p,q}(\rd)\|.
$$
\end{proposition}

\section{Embeddings with smoothness $\frac{p}{u}\sigma_p$}

We return to the remarkable coincidence \eqref{mazzu} and consider the limiting situation when $p=1$. Recall that in case of $p=u=1$ it is well-known that $F^0_{1,2}(\rd) \hookrightarrow L_1(\rd)$ properly embedded, cf. \cite{ST95}. Now we concentrate on the Morrey situation when $p=1<u$ and can prove some partial counterpart of \eqref{mazzu}.

\begin{proposition} \label{prop-Eu12}
Let $1<u<\infty$. Then  
\[
\cE^0_{u,1,2}(\rd)\hookrightarrow \cM_{u, 1}(\rd) .
\]
\end{proposition}

\begin{proof}
\emph{Step 1.}  We prove that there exists a positive constant $C>0$ such that
\begin{equation}\label{0u11}
\|f|\cM_{u,1}(\rd)\| \le C\,\|f|\cE^0_{u,1,2}(\rd)\|\, .  
\end{equation} 
Let $f\in \cE^0_{u,1,2}(\rd)$. By the atomic decomposition theorem with atoms satisfying the conditions from Definition \ref{defatoms} with  $N>\whole{d(u-1)}$ and $c_1=3$,  we have
\begin{equation} \nonumber
f = \sum_{j=0}^\infty \sum_{m\in \zd} \lambda_{jm} a_{jm} 
\end{equation} 
with 
\begin{equation} \nonumber
\Big\| \Bigl(\sum_{j=0}^\infty \sum_{m\in \zd} |\lambda_{jm}|^2 \, \chi_{jm}(\cdot )\Bigr)^{\frac{1}{2}}  \big| \cM_{u,1}(\rd)\Big\| \le c \|f|\cE_{u,1,2}^0(\rd)\|. 
\end{equation}
Let $Q=Q_{\nu,k}$ be a dyadic cube and  decompose $f$ as follows
\begin{equation} \label{decomp-f}
f = \sum_{j=0}^{\nu} \sum_{m\in \zd} \lambda_{jm} a_{jm} +  \sum_{j=\nu+1}^\infty \sum_{m\in \zd} \lambda_{jm} a_{jm} =f_1+f_2.
\end{equation}
Remark that $f_i\in \cE_{u,1,2}^0(\rd)$ and $ \|f_i|\cE_{u,1,2}^0(\rd)\| \le  c \|f|\cE_{u,1,2}^0(\rd)\|$, $i=1,2$.

\smallskip

\emph{Step 2.}  We deal first with $f_1$. Recall $Q=Q_{\nu,k}$. Let  $\mu_{j,\nu}=\{m\in\zd: Q\cap\supp a_{j,m} \not= \emptyset\}$ and note that $\mu_{j,\nu}\leq c$ if  $ j\leq \nu$. Then we have
\begin{align}
 & |Q|^{\frac{1}{u}-1} \int_Q |f_1(x)| \dint x  \leq    |Q|^{\frac{1}{u}-1}  \sum_{j=0}^{\nu} \sum_{m\in \mu_{j,\nu}} |\lambda_{jm}| \int_Q |a_{jm}(x)| \dint x  \nonumber \\
 &  \leq c |Q|^{\frac{1}{u}}  \sum_{j=0}^{\nu}  \sum_{m\in \mu_{j,\nu}}   |\lambda_{jm}| = c |Q|^{\frac{1}{u}}  \sum_{j=0}^{\nu}  \sum_{m\in \mu_{j,\nu}}  |Q_{jm}|^{-1} \int_{Q_{jm}}\Bigl(|\lambda_{jm}|^2 \chi_{jm}(x) \Bigr)^{\frac{1}{2}} \dint x \nonumber \\
 &  \leq c  \sum_{j=0}^{\nu} 2^{(j-\nu)\frac{d}{u}} \sum_{m\in \mu_{j,\nu}}  |Q_{jm}|^{\frac{1}{u}-1} \int_{Q_{jm}}\Bigl(\sum_{\ell=1}^{\infty} \sum_{n\in\zd}|\lambda_{\ell n}|^2 \chi_{\ell n}(x) \Bigr)^{\frac{1}{2}} \dint x \nonumber 
 \end{align}
 \begin{align}
 &  \leq c  \Bigl(\sum_{j=0}^{\nu} 2^{(j-\nu)\frac{d}{u}}  \Bigr)\sup_{0\leq j\leq \nu} \sup_{m\in  \mu_{j,\nu} }\Big\{  |Q_{jm}|^{\frac{1}{u}-1} \int_{Q_{jm}}\Bigl(\sum_{\ell=1}^{\infty} \sum_{n\in\zd}|\lambda_{\ell n}|^2 \chi_{\ell,n}(x) \Bigr)^{\frac{1}{2}} \dint x \Big\}\nonumber \\
 & \leq   c \Big\| \Bigl(\sum_{j=0}^\infty \sum_{m\in \zd} |\lambda_{j m}|^2 \, \chi_{j,m}(\cdot )\Bigr)^{\frac{1}{2}}  \big| \cM_{u,1}(\rd)\Big\| \le c \|f|\cE_{u,1,2}^0(\rd)\|. \label{est-f1}
\end{align}

\smallskip

\emph{Step 3.} Now we deal with $f_2$ for what we rely on the results in \cite{ho}. By checking the proof in \cite[Thm.~4.3]{ho}, we can see that $f_2 \in \cE^0_{u,1,2}(\rd)$ can be decomposed in terms of non-smooth atoms $b_{jm}$ supported in cubes $cQ_{jm}$ with side lengths less than $2^{-\nu}$,  that is,   
$$
 f_2=\sum_{j=\nu+1}^\infty \sum_{m\in \zd} \lambda_{jm} a_{jm} = \sum_{j=\nu+1}^\infty \sum_{m\in \zd}  t_{jm} b_{jm}
$$
where
$$
\| t\vert m_{u,1}\|:=\sup_{P\in\mathcal{Q}}|P|^{\frac{1}{u}-1}\Bigl(\sum_{Q_{\ell k}\subset P}|t_{\ell k}| \Bigr)\leq c \|f_2|\cE_{u,1,2}^0(\rd)\| 
$$
and $t=\{t_{jm}\}_{j\in\no,m\in\zz^d}$. 

We remark that the non-smooth atoms from \cite{ho} differ from the smooth ones of Definition~\ref{defatoms} in condition \eqref{cond-atoms}; in particular, we have 
$$
\|b_{jm}\vert L_u(\rd)\|\leq |Q_{jm}|^{\frac{1}{u}-1},  \quad j\in\no,\; m\in\zz^d.
$$ 
Then, using H\"older's inequality and the properties of non-smooth atoms we obtain
\begin{align}
  |Q|^{\frac{1}{u}-1} &\int_Q |f_2(x)| \dint x    \leq  |Q|^{\frac{1}{u}-1} \sum_{j=\nu +1}^{\infty} \sum_{m\in \zd} |t_{jm}|  \int_{3Q_{jm}\cap Q} |b_{jm}(x)| \dint x  \nonumber \\
    &  \leq c  |Q|^{\frac{1}{u}-1} \sum_{j=\nu +1}^{\infty} \sum_{m\in \zd: Q_{jm}\subset cQ} |t_{jm}||Q_{jm}|^{1-\frac{1}{u}}  \|b_{jm}\vert L_u(\rd)\|  \nonumber \\
        &  \leq  c |Q|^{\frac{1}{u}-1} \sum_{j=\nu +1}^{\infty} \sum_{m\in \zd: Q_{jm}\subset c' Q} |t_{jm}| \leq  c \|(t_{jm})\vert m_{u,1}\| \nonumber \\
& \le c \|f|\cE_{u,1,2}^0(\rd)\|. \label{est-f2}
\end{align}
The desired outcome is then a consequence of \eqref{decomp-f}, \eqref{est-f1} and \eqref{est-f2}.
\end{proof}


\medskip

 We equip the spaces $ \Lloc(\rd)$ with the metric 
 \[ d(f,g)= \sum_{n=1}^\infty \frac{1}{2^n} \frac{\|f-g|L_1(\widetilde{Q}_n)\|}{1+\|f-g|L_1(\widetilde{Q}_n))\|}, \]
 where $\widetilde{Q}_n = [-n,n]^d$.  The space $ \Lloc(\rd)$ with this metric  is a complete locally convex metric space, i.e., a Fr\'echet space, cf. \cite[page 40]{MV}. One can easily see that a Morrey space $\cM_{u,p}(\rd)$ is continuously embedded into $ \Lloc(\rd)$ if $p\ge 1$. Indeed we have
 \begin{align*}
 d(f,g) 
 &\le C \sum_{n=1}^\infty \frac{1}{2^n} \frac{|\widetilde{Q}_n|^{1-\frac{1}{u}}}{1+\|f-g|L_1(\widetilde{Q}_n)\|}\|f-g|\cM_{u,p}(\rd)\| \le\\
 &\le C \sum_{n=1}^\infty \frac{|\widetilde{Q}_n|^{1-\frac{1}{u}}}{2^n }\|f-g|\cM_{u,p}(\rd)\| \le C \|f-g|\cM_{u,p}(\rd)\|.
 \end{align*}

\begin{theorem} \label{theorem-1-e}
Let  $0<p \leq u <\infty$,   $0<q\leq\infty$,  and $ s=\frac{p}{u}\sigma_p$. The following assertions are equivalent:
\begin{itemize}
\item[{\upshape\bfseries (i)}] $\cE^s_{u,p,q}(\rd)\hookrightarrow \Lloc(\rd) $,
\item[{\upshape\bfseries (ii)}] $\cE^s_{u,p,q}(\rd) \hookrightarrow \cM_{\frac{u}{\min(p,1)},\max(p,1)}(\rd) $,
\item[{\upshape\bfseries (iii)}] either \ $p\geq 1$ and $q\leq 2$, or \ $0<p<1$. 
\end{itemize}
\end{theorem}

\begin{proof} Note that the implication   (ii)  $\Rightarrow$ (i) is already shown, and the part 
(i)  $\Rightarrow$ (iii) is covered by \cite[Thm.~3.4]{hms}.

\emph{Step 1}.  Consider first the case  $p=1$. The implication  (iii)  $\Rightarrow$ (ii) is a consequence of Proposition \ref{prop-Eu12} and an elementary embedding, 
\[
\cE^0_{u,1,q}(\rd) \hookrightarrow \cE^0_{u,1,2}(\rd)  \hookrightarrow \cM_{u,1}(\rd) \quad \text{if} \quad 0<q\leq 2 .
\]
 
 \smallskip

\emph{Step 2}.  Let  $0<p<1$.  Then $s=\frac{p}{u}\sigma_p = \frac{d}{u}(1-p)$. The general properties of embeddings between Triebel-Lizorkin-Morrey spaces, cf. \cite[Thm.~3.1]{hs13}, and the first step imply   
\begin{equation} \nonumber
\cE^s_{u,p,\infty} (\rd) \hookrightarrow \cE^0_{\frac{u}{p},1,1} (\rd) \hookrightarrow \cM_{\frac{u}{p},1}(\rd),  
\end{equation} 
which shows that  (iii)  $\Rightarrow$ (ii), 
and there is nothing more to be proved in this case.

 \smallskip
 
\emph{Step 3}.  Now assume $p>1$. The implication  (iii)  $\Rightarrow$ (ii) follows from 
\[
\cE^0_{u,p,q}(\rd) \hookrightarrow \cE^0_{u,p,2}(\rd) =\cM_{u,p}(\rd) \quad \text{if} \quad 0<q\leq 2 .
\]
\end{proof}

\begin{remark}
The above theorem improves the statements of \cite[Thm.~3.4]{hms}  and extends Theorem 3.3.2(i) and Corollary 3.3.1 in \cite{ST95} from classical Triebel-Lizorkin spaces to Triebel-Lizorkin-Morrey spaces.
\end{remark}

\begin{theorem} \label{theorem-n_290319}
Let  $0<p \leq u <\infty$,   $0<q\leq\infty$,  and $ s=\frac{p}{u}\sigma_p$. The following assertions are equivalent:
\begin{itemize}
\item[{\upshape\bfseries (i)}] $\cN^s_{u,p,q}(\rd) \hookrightarrow  \Lloc(\rd) $,
\item[{\upshape\bfseries (ii)}] $\cN^s_{u,p,q}(\rd) \hookrightarrow \cM_{\frac{u}{\min(p,1)},\max(p,1)}(\rd) $,
\item[{\upshape\bfseries (iii)}] $0<q \leq \min \bigl(\max(p,1),2 \bigr)$. 
\end{itemize}
\end{theorem}

\begin{proof} The case $p=u$ is well known, cf. \cite[Thm.~3.3.2, Cor.~3.3.1]{ST95}, so we can restrict ourselves to the case $p<u$.

\emph{Step 1}. We prove that (iii) $\Rightarrow$ (ii) $\Rightarrow$ (i). The second implication  has been already shown so it remains to prove the first one. 

 Let $0<p\le 1$. For $0<q\leq 1$, by  general properties of embeddings, in particular \cite[Thm.~3.2]{hs12},  and Theorem \ref{theorem-1-e},  we have
\[
\cN^s_{u,p,q} (\rd) \hookrightarrow \cN^0_{\frac{u}{p},1,1} (\rd) \hookrightarrow \cE^0_{\frac{u}{p},1,1}(\rd) \hookrightarrow \cM_{\frac{u}{p},1}(\rd)
\]
which  proves the implication (iii) $\Rightarrow$ (ii).

Consider the case $p\geq 1$.  If $0<q \leq \min(p,2)$, then elementary embeddings, \eqref{mazzu} and Theorem~\ref{theorem-1-e} yield 
\[
\cN^0_{u,p,q}(\rd) \hookrightarrow \cE^0_{u,p,2}(\rd)  \hookrightarrow \cM_{u,p}(\rd),
\]
which shows  that  (iii)  $\Rightarrow$ (ii).


%
\emph{Step 2}. We prove that the condition (i) implies (iii). For $0<p<1$ the implication  was proved in \cite{hms}, cf. Step 2 of the proof of Theorem 3.4. But the same argument works for $p=1$. Also the case  $2\le p<\infty$ is covered by Theorem 3.4 in \cite{hms}. 

 It remains to consider the case $1<p<2$.  We assume that embedding (i) holds for some $q>p$.  We choose a smooth  function    $\tilde{a}$ such that 
\begin{equation}
  \supp \tilde{a} \subset \left[0,\frac{1}{2}\right]^d, \qquad  0\le \tilde{a} \le 1 \qquad \text{and} \qquad   \left|\frac{\partial \tilde{a}}{\partial x_i}\right|\le 1,   \quad i  =1,\dots, d,
\end{equation}
and put 
\begin{equation}
a(x)= \tilde{a}(x) - \tilde{a}(-x). 
\end{equation}
Then $a$ is an atom satisfying the first moment condition and supported in $[-\frac{1}{2},\frac{1}{2}]^d$. Moreover we consider the family of atoms $a_{j,m}$, $j=0,1,\ldots$ and $m\in \mathbb{Z}^d$ that are the dilations and translations of $a$, i.e.,  
\[ a_{j,m}(x)=a(2^{j}x-m-(1/2,\ldots ,1/2)).  \] 
The function $a_{j,m}$ is an atom supported in $Q_{j,m}$. 

Let us fix $n\in \mathbb{N}$. For any cube $Q_{j,m} \subset  Q_{0,0}$, $1\le j\le n$, $m\in\zd$, we define a function 
\[ b_{j,m}(x) = 2^{j-n}\sum_{k: Q_{n,k}\subset Q_{j,m}} a_{n,k}(x).\]
The function  $b_{j,m}$ is an atom satisfying the first moment condition supported in $Q_{j,m}$. 
We define a smooth  function $f_n$ by the following finite sum
\[f_n(x) = \sum_{j=1}^n j^{-\frac{1}{p}} \sum_{m: Q_{j,m} \subset  Q_{0,0}}b_{j,m}(x).\] 
The functions $f_n$ belong to  $\cN^0_{u,p,q}(\rd)$ and their  norms are uniformly bounded since 
\begin{align*}
\|f_n|\cN^0_{u,p,q}(\rd)\| & \le \left( \sum_{j=1}^n 2^{-jq\frac{d}{u}} \sup_{\nu:\nu \le j ,k\in \mathbb{Z}^d} 2^{dq(j-\nu)(\frac{1}{u}-\frac{1}{p})}  \Big(\sum_{Q_{j,m}\subset Q_{\nu,k}\subset Q_{0,0}} j^{-1} \Big)^\frac{q}{p} \right)^\frac{1}{q} \\
& \le \left( \sum_{j=1}^n  j^{-\frac{q}{p}}2^{-jq\frac{d}{u}} \sup_{\nu:0\leq \nu \le j ,k\in \mathbb{Z}^d} 2^{dq\frac{(j-\nu)}{u}}\right)^\frac{1}{q} \\
&  \le 
\left( \sum_{j=1}^\infty  j^{-\frac{q}{p}}\right)^\frac{1}{q}= C < \infty , 
\end{align*} 
recall $q>p$. 
On the other hand, for $\nu\in\nn$,  
\begin{align*}
  \|f_n|L_1(\widetilde{Q}_\nu)\| &= \|f_n|L_1(\widetilde{Q}_1)\| = \int_{\widetilde{Q}_1} \left| \sum_{j=1}^n j^{-\frac{1}{p}} 
  \sum_{Q_{j,m} \subset  Q_{0,0}}
b_{j,m}(x)\right| \dint x \\ 
& = \int_{\widetilde{Q}_1} \left| \sum_{j=1}^n j^{-\frac{1}{p}} 2^{j-n}\sum_{Q_{j,m} \subset  Q_{0,0}}
 \sum_{Q_{n,k}\subset Q_{j,m}} a_{n,k}(x)\right| \dint x \\
& =  \int_{\widetilde{Q}_1} \left| \sum_{j=1}^n j^{-\frac{1}{p}} 2^{j-n} n \sum_{Q_{n,k} \subset  Q_{0,0}}
  a_{n,k}(x)\right| \dint x  \\ 
& = \int_{\widetilde{Q}_1} \left( \sum_{j=1}^n j^{-\frac{1}{p}} 2^{j-n} n\right)\left| \sum_{Q_{n,k} \subset  Q_{0,0}}
  a_{n,k}(x)\right| \dint x  \\ 
 &=   \sum_{j=1}^n j^{-\frac{1}{p}} 2^{j-n} n \sum_{Q_{n,k} \subset  Q_{0,0}}
  \int_{\widetilde{Q}_1} | a_{n,k}(x)| \dint x  \\
  & = c n 2^{-n}  \sum_{j=1}^n j^{-\frac{1}{p}} 2^j   \ge c  n^{1-\frac{1}{p}}. 
  \end{align*}
Thus $\|f_n|L_1(\widetilde{Q}_\nu)\|= \|f_n|L_1(\widetilde{Q}_1)\|\rightarrow \infty$ if $n\rightarrow \infty$ since $1<p$. The local base at zero in $ \Lloc(\rd)$ is given by the sets $V_{k,\varepsilon} = \{f: \|f|  L_1(\widetilde{Q}_\nu)\|<\varepsilon,\quad  \nu=1,\ldots , k\}$. So the sequence $f_n$ is not bounded in $ \Lloc(\rd)$. This contradicts  the continuity of the embedding (i).  
\end{proof}

\begin{corollary}\label{theorem-2-n}
Let $0<p\le u < v<\infty$, $1\le q \le v$  and     $s=\frac{d}{u}- \frac{d}{v}$.  
Then the following assertions  are equivalent: 
\begin{itemize}
\item[{\upshape\bfseries (i)}] $\cE^s_{u,p,\infty}(\rr^d)\hookrightarrow \cM_{v,q}(\rr^d) $,
\item[{\upshape\bfseries (ii)}] $\cN^s_{u,p,q}(\rr^d)\hookrightarrow \cM_{v, q}(\rr^d)    $,
\item[{\upshape\bfseries (iii)}] $q\le v\frac{p}{u}$. 
\end{itemize}
\end{corollary}

\begin{proof} 
\emph{Step 1.} First we prove the sufficiency of the condition  {\upshape\bfseries (iii)}. 
If $p= \frac u v$, then $\frac{p}{u}\sigma_p = \frac{d}{u}- \frac{d}{v}=s$. So {\upshape\bfseries (i)} 
follows from Theorem   \ref{theorem-1-e}. If $p > \frac u v$, then $s= \frac{d}{u}- \frac{d}{v}>0$ and $q_{\max}=\frac{pv}{u}>1$. In that case the embedding is a consequence of Sobolev embeddings \cite[Thm.~3.2]{hs13},  the Paley-Littlewood formula \eqref{mazzu} and the embeddings between Morrey spaces.
   
Similarly, the embedding {\upshape\bfseries (ii)} 
in the case $p= \frac u v$  follows from Theorem \ref{theorem-n_290319} since  $q= 1$.  To prove {\upshape\bfseries (ii)} 
for $p > \frac u v$ and $1<q\leq p\frac v  u$ one can use the Franke-Jawerth embeddings for smoothness Morrey spaces, cf. \cite[Thm.~4.3]{hs13}. Indeed,  we have
\[
\cN^s_{u,p,q}(\rr^d)\hookrightarrow  \cE^0_{v,q,2}(\rr^d) = \cM_{v,q}(\rr^d). 
\]
If $p>\frac u v$ and $q=1$, then 
\[\cN^s_{u,p,q}(\rr^d)\hookrightarrow  \cN^s_{u,\frac{u}{v},q}(\rr^d) \hookrightarrow \cM_{v,1}(\rr^d),\] 
the statement coincides with our previous statement, Theorem~\ref{theorem-n_290319}. We refer to \cite{hs12} for a proof of the first embedding. 

\emph{Step 2.} Now we prove the necessity of the conditions. 

 First we assume that
\begin{align}
\cN^s_{u,p,q}(\rr^d)\hookrightarrow \cM_{v, q}(\rr^d).    \label{cor-n-hip}
\end{align} 
The last embeddings implies that the space  $\cN^s_{u,p,q}(\rr^d)$ consists of locally integrable functions,  so $s\ge \frac{p}{u}\sigma_p$, cf. Theorem  3.3. in \cite{hms}. This implies $ p\ge \frac{u}{v}$. 
 If $1<q$, 
 then  by   \eqref{elem} and \eqref{mazzu} we get  
\[
\cN^s_{u,p,q}(\rd)\hookrightarrow \cM_{v, q}(\rr^d)= \cE^0_{v, q,2}(\rd) \hookrightarrow \cN^0_{v, q,\infty}(\rd).
\]
But  Theorem 3.3 of \cite{hs12}, implies $q\le \frac{pv}{u}$. If $q= 1$, then the condition $q\leq \frac{pv}{u}$ is satisfied automatically since we have already proved that $p\ge \frac{u}{v}$.

Using the same argument as above we can prove that    $q\le \frac{pv}{u}$ if
\[\cE^s_{u,p,\infty}(\rr^d)\hookrightarrow \cM_{v, q}(\rr^d) .\]
\end{proof}

\begin{remark}\label{at-into-M}
The embeddings {\upshape\bfseries (i)}  and {\upshape\bfseries (ii)} of the last corollary holds also for $q< 1$ if $ v\frac{p}{u}\ge 1 $ since then we increase the target space whereas the source space is unchanged or smaller. 
\end{remark}
\begin{remark} 
Another class of generalisations of Besov and Triebel-Lizorkin spaces are Besov-type spaces $\bt(\rd)$ and Triebel-Lizorkin-type spaces $\ft(\rd)$, $0 < p,q \leq\infty$ (with $p<\infty$ in case of $\ft$), $\tau\ge 0$, $s \in \real$  introduced in \cite{ysy}. The spaces are strictly related to $\MB(\rd)$ and $\MF(\rd)$ spaces if $0\le \tau< \frac{1}{p}$. Namely 
\begin{align}
\ft(\rd)= \MF(\rd) \qquad\text{if} \qquad 0\le \tau = \frac{1}{p}-\frac{1}{u}< \frac{1}{p},\label{Ft-E}
\intertext{and} 
\MB(\rd) \hookrightarrow \bt(\rd)\qquad\text{if} \qquad 0\le \tau = \frac{1}{p}-\frac{1}{u}< \frac{1}{p}.\label{Bt-N} 
\end{align}
The Besov-Morrey and Besov-type spaces coincide only if $\tau=0$ or $q=\infty$. In contrast to the spaces $\MA(\rd)$ these scales are embedded into each other like their classical counterparts, that is,
\begin{equation}\label{elem-tau}
	B^{s,\tau}_{p,\min(p,q)}(\rd)\, \hookrightarrow \, \ft(\rd)\, \hookrightarrow \,B^{s,\tau}_{p,\max(p,q)}(\rd),
\end{equation}
whenever $0<p<\infty$, $0<q\leq\infty$, $s\in\real$, $\tau\geq 0$. 
Using these relations we can easily transfer our results to the new class of function spaces if $0\le \tau<\frac{1}{p}$. In particular, if $0\leq \tau<\frac{1}{p}$ and $s=(1-\tau p)\sigma_p$, then the following three conditions are equivalent
\begin{itemize}
\item[{\upshape\bfseries (i)}] $\ft(\rd) \hookrightarrow \Lloc(\rd) $,
\item[{\upshape\bfseries (ii)}] $\ft(\rd) \hookrightarrow \cM_{v,\max(p,1)}(\rd) $, where $v=\frac{p}{(1-\tau p)\min(p,1)}$, 
\item[{\upshape\bfseries (iii)}] either \ $p\geq 1$ and $q\leq 2$, or \ $0<p<1$. 
\end{itemize}
The above statement improves Theorem 3.8 in \cite{hms} and corrects some misprint concerning non-embeddings of $\ft(\rd)$ spaces stated in formula (3.31) there.
  
Moreover we have the following counterpart of Corollary~\ref{theorem-2-n}. If
\begin{equation}
  0\leq \tau<\frac1p, \quad \frac{p}{1-p\tau}<v<\infty, \quad s=d\left(\frac{1}{p} - \tau\right)- \frac{d}{v}, \quad  
  1\le q \leq v,\label{At-M-ass}
  \end{equation}
  then
 \begin{equation}
   F^{s,\tau}_{p,\infty}(\rd) \hookrightarrow \cM_{v,q}(\rd)\quad \iff \quad q\le v(1- p\tau). \label{Ft-M}
\end{equation}
This is a direct consequence of \eqref{Ft-E} and Corollary~\ref{theorem-2-n}. In case of spaces $\bt(\rd)$ we have a partial counterpart only at the moment: assume \eqref{At-M-ass} and $q \leq p$. Then
  \begin{equation}
B^{s,\tau}_{p,q}(\rd) \hookrightarrow \cM_{v,q}(\rd), \label{Bt-M}
 \end{equation}   
as by assumption, $q\leq p<v(1-p\tau)$; thus \eqref{Ft-M} together with \eqref{elem-tau} conclude the argument. In case of $q>p$ the situation is not yet complete: while the embedding \eqref{Bt-N} together with Corollary~\ref{theorem-2-n} always lead to $q\leq v(1-p\tau)$ whenever $\bt(\rd)\hookrightarrow \cM_{v,q}(\rd)$, the sufficiency is not clear in all cases: assume, in addition to \eqref{At-M-ass} that $p<q\leq \frac{v}{\tau v+1}$. Then using some Franke-Jawerth embedding, cf. \cite[Thm.~3.10]{YHMSY},
  \[
  \bt(\rd)\hookrightarrow F^{\sigma,\tau}_{q,\infty}(\rd),\quad \sigma = s-\frac{d}{p}+\frac{d}{q}<s,
  \]
  and, by \eqref{Ft-M} again,
  \[
  F^{\sigma,\tau}_{q,\infty}(\rd) \hookrightarrow \cM_{v,q}(\rd)\quad \text{if}\quad q\leq v(1-q\tau),
  \]
which is satisfied by our additional assumption on $q$, we get \eqref{Bt-M}.   
\end{remark}

\section{Embeddings with smoothness $\frac{d}{u}$}
In this section we are interested in embeddings of spaces with smoothness $s=\frac d u$. This is once more the borderline smoothness since if the smoothness is strictly bigger than $\frac d u$,  then the space consists of bounded functions. We describe the properties of the functions belonging to the spaces with smoothness $s=\frac d u$ in terms of exponential Orlicz-Morrey spaces and some  generalised Morrey spaces. 

Our approach is based on the extrapolation argument that in the case of  Besov and Triebel-Lizorkin spaces was elaborated by Triebel in \cite{Tri93}. We follow the general idea of his work.   The extrapolation inequalities we need are formulated in the following lemma.

\begin{lemma} \label{lemma1}
Let $0< p\le u<\infty$, $0<q\le \infty$, $p<r<\infty$ and $v=\frac{u r}{p}$. 
 \begin{itemize}
\item[{\upshape\bfseries (i)}] If $r\ge 1$, then  there is a constant $C>0$ depending on $d$, $p$, $u$ and $q$ but independent of $v$, $r$ such that for any $f\in \MBd(\rd)$ the following inequalities hold
\begin{align}\label{n1}
\|f|\cM_{v,r}(\rd)\| & \le C v^{1-\frac{1}{q}} \|f|\MBd(\rd)\| \qquad \text{if}\quad q\ge 1, \\
\|f|\cM_{v,r}(\rd)\| & \le C  \|f|\MBd(\rd)\| \qquad \text{if}\quad q\le 1. \label{n2}
\end{align}
\item[{\upshape\bfseries (ii)}] If $p<r< 1$, then  there is a constant $C>0$ depending on $d$, $p$, $u$  but independent of $v$ and $r$  such that for any $f\in \mathcal{N}^{\frac{d}{u}}_{u,p,\infty}(\rd)$ the following inequality holds
\begin{align}\label{n3}
\|f|\cM_{v,r}(\rd)\|  \le C  \|f|\mathcal{N}^{\frac{d}{u}}_{u,p,\infty}(\rd)\| . 
\end{align}
\end{itemize}
\end{lemma}

 \begin{remark}
The constant $C$ is independence of $v$ in that sense that it depends on $u$ and $p$ but takes the same value as far as the  the quotient $\frac{u}{p}$ is constant. 
\end{remark} 


\begin{proof}
First we prove that 
\begin{equation}
|\psi(x)|\le C \|\psi|\cM_{u,p}(\real^d)\| 
\end{equation}
if $\supp \mathcal{F}\psi\subset B(0,2)$ and $0<p\le u<\infty$. The proof is standard and based on the Plancherel-Polya-Nikol'skii   inequality  
\begin{equation}
\sup_{y\in \real^d} \frac{|\psi(x-y)|}{(1+|y|)^{d/\delta}} \le C \Big(M(|\psi|^\delta)(x)\Big)^\frac{1}{\delta}, \qquad x\in\real^d,\; \delta>0,
\end{equation}
cf. \cite[Theorem 1.3.1]{T-F1}. Here $M$ stands for the Hardy-Littlewood maximal operator, as usual.

We repeat the argument for completeness, cf. also \cite{nns}. 
One can easily prove that if $|x-y|\le 2$, then 
\begin{equation} 
|\psi(x)| \le c \sup_{z\in \real^d} \frac{|\psi(z)|}{(1+|z-y|)^{d/\delta}} \le C \Big(M(|\psi|^\delta)(y)\Big)^\frac{1}{\delta}.  
\end{equation}
So, if $\delta<p$, then the boundedness of the maximal function in Morrey spaces gives us
\begin{align}
|\psi(x)| &\le c \bigg( |B(x,2)|^{-1} \int_{B(x,2)} \Big(M(|\psi|^\delta)(y)\Big)^\frac{p}{\delta} \dint y\bigg)^{1/p} \nonumber\\
& \le C \left\|\Big(M(|\psi|^\delta)\Big)^\frac{1}{\delta}|\cM_{u,p}(\rd)\right\|
= C 
\left\|M(|\psi|^\delta)|\cM_{u/\delta,p/\delta}(\real^d)\right\|^{\frac{1}{\delta}} \nonumber \\ 
& \le C \left\||\psi|^\delta|\cM_{u/\delta,p/\delta}(\rd)\right\|^{\frac{1}{\delta}} =  C  \|\psi|\cM_{u,p}(\rd)\|, \qquad  x\in \rd.\nonumber 
\end{align}
Thus, if $p<r$, then for any cube $Q$ we have 
\begin{align}
\Big(\int_Q |\psi(x)|^r \dint x\Big)^\frac{1}{r} \le \sup_{x\in Q} |\psi(x)|^{1-\frac{p}{r}} \Big(\int_Q |\psi(x)|^p \dint x \Big)^\frac{1}{r} \le  
 C |Q|^{\frac{1}{r}-\frac{p}{ru}}
\|\psi|\cM_{u,p}(\rd)\|.  \nonumber
\end{align}
So 
\begin{equation}
\|\psi|\cM_{v,r}(\rd)\| \le C \|\psi|\cM_{u,p}(\rd)\|.
\end{equation}
The last inequality implies 
\begin{equation}
\|\mathcal{F}^{-1}\big(D_{2^j}\varphi_j \mathcal{F}f\big)| \cM_{v,r}(\real^d)\| \le C \|\mathcal{F}^{-1}\big(D_{2^j}(\varphi_j \mathcal{F}f)\big)|\cM_{u,p}(\rd)\|
\end{equation}
since $\supp D_{2^j}(\varphi_j \mathcal{F}f) \subset B(0,2)$, where  $D_\delta g(x) = g(\delta x)$. 
Now the formula for  Fourier dilations and the relation between dilations and the Morrey norms give us 
\begin{equation}
\|\mathcal{F}^{-1}\big(\varphi_j \mathcal{F}f\big)| \cM_{v,r}(\real^d)\| \le C 2^{jd(\frac{1}{u}-\frac{1}{v})} \|\mathcal{F}^{-1}\big(\varphi_j \mathcal{F}f\big)|\cM_{u,p}(\rd)\|.
\end{equation}
In consequence, if $1\le r$ and $1<q\le\infty$ we have 
\begin{align}
\|f|\cM_{v,r}(\real^d)\| & \le \sum_{k=0}^\infty \|\mathcal{F}^{-1}\varphi_k \mathcal{F}f| \cM_{v,r}(\rd)\| \nonumber
\\
& \le  C  \sum_{k=0}^\infty 2^{kd(\frac{1}{u}-\frac{1}{v})}\|\mathcal{F}^{-1}\varphi_k \mathcal{F}f| \cM_{u,p}(\rd)\|\nonumber\\
 & \le C \Big(\sum_{k=0}^\infty 2^{-kd\frac{q'}{v}}\Big)^\frac{1}{q'}\Big(\sum_{k=0}^\infty 2^{k\frac{d}{u}q}\|\mathcal{F}^{-1}\varphi_k \mathcal{F}f| \cM_{u,p}(\real^d)\|^q\Big)^\frac{1}{q} 
 \nonumber 
 \\  & \le C v^{1-\frac{1}{q}} \|f|\MBd(\rd)\|, \nonumber 
\end{align}
 with obvious changes if $q=\infty$; similarly,  
\[
 \|f|\cM_{v,r}(\real^d)\|  \le C \sum_{k=0}^\infty 2^{kd(\frac{1}{u}-\frac{1}{v})}\|\mathcal{F}^{-1}\varphi_k \mathcal{F}f| \cM_{u,p}(\real^d)\| \le C \|f|\MBd(\rd)\|
\]
if $1\le r$ and $0<q\le 1$. 

Now let  $p< r< 1$ and  $q=\infty$.  We have 
\begin{align}
\|f|\cM_{v,r}(\rd)\|^r & \le \sum_{k=0}^\infty \|\mathcal{F}^{-1}\varphi_k \mathcal{F}f| \cM_{v,r}(\rd)\|^r \nonumber \\
& \le  C\sum_{k=0}^\infty 2^{kdr(\frac{1}{u}-\frac{1}{v})}\|\mathcal{F}^{-1}\varphi_k \mathcal{F}f| \cM_{u,p}(\rd)\|^r  \nonumber \\
& \le C \Big(\sup_{k\in\nn_0}2^{k\frac{d}{u}}\|\mathcal{F}^{-1}\varphi_k \mathcal{F}f| \cM_{u,p}(\rd)\|\Big)^r\ \sum_{k=0}^\infty 2^{-kd\frac{r}{v}} .
\nonumber 
\end{align}
But 
\[
\Big(\sum_{k=0}^\infty 2^{-kd\frac{r}{v}} \Big)^\frac{1}{r} \le C_{u,p} <\infty, 
\]
so 
\begin{equation} \nonumber
\|f|\cM_{v,r}(\real^d)\|  \le C \|f|\mathcal{N}^{\frac{d}{u}}_{u,p,\infty}(\rd)\|.
\end{equation}
This concludes the argument.
\end{proof}

\subsection{Embeddings in Orlicz-Morrey spaces}

The Orlicz-Morrey spaces  considered below were introduced by Nakai \cite{nakai}. They are a generalisation of both Morrey and Orlicz spaces.

\begin{definition} Let $\Phi:[0,\infty) \rightarrow [0, \infty)$ be a Young function, i.e., a continuous convex function with $\Phi(0)=0$ and $\lim_{t\rightarrow \infty}\Phi(t)=\infty$. For $1\leq r<\infty$ and a cube $Q$ we put
\[
\|f\|_{(r,\Phi);Q}:= \inf \left\{  \lambda>0:  |Q|^{\frac{1}{r}-1}\int_Q\Phi\Bigl( \frac{|f(x)|}{\lambda}\Bigr)\dint x\leq 1 \right\} .
\]
The Orlicz-Morrey space 
$\cM_{r,\Phi}(\rd)$ is the set of all measurable functions $f$ such that 
\[\|f\vert \cM_{r,\Phi}(\rd)\|:= \sup_{Q\in\mathcal{Q}}\|f\|_{(r,\Phi);Q} < \infty . \]
\end{definition}

We consider also the following  expression
$$
\|f\vert \cM_{r,\Phi}(\rd)\|^*:=\inf \left\{  \lambda>0: \sup_{Q\in\mathcal{Q}} |Q|^{\frac{1}{r}-1}\int_Q\Phi\Bigl( \frac{|f(x)|}{\lambda}\Bigr)\dint x\leq 1 \right\}.
$$
Since $ \sup_Q\inf_\lambda \le \inf_\lambda\sup_Q$ we have $\|f\vert \cM_{r,\Phi}(\rd)\|\le \|f\vert \cM_{r,\Phi}(\rd)\|^*$. In the next lemma we  show that if the Young function $\Phi$ is of exponential type, then both   expressions are  equivalent and the space  $\cM_{r,\Phi}(\rd)$  can be characterised by extrapolation. 

\begin{lemma} \label{lemma-equiv-OM}
Let $0< p\le u<\infty$, $0<q < \infty$, and $\Phi_{p,q}(t):=t^p \exp ( t^q)$. Then $f$ belongs to the Orlicz-Morrey space $\cM_{\frac{u}{p},\Phi_{p,q}}(\rd)$ if, and only if,
$$
\sup_{j\geq 1}j^{-1/q}\, \|f\vert \cM_{v(j),p(j)}(\rd)\|<\infty,
$$
where $p(j)=p+jq$ and $v(j)=\frac{u}{p}p(j)$. Moreover, 
$$
\|f\vert \cM_{\frac{u}{p},\Phi_{p,q}}(\rd)\| \sim  \|f\vert \cM_{\frac{u}{p},\Phi_{p,q}}(\rd)\|^* \sim \sup_{j\geq 1}j^{-1/q}\, \|f\vert \cM_{v(j),p(j)}(\rd)\|. 
$$
\end{lemma}

\begin{proof} 
\emph{Step 1}. It is sufficient to prove that there are constants $c,C>0$ such that for any measurable function $f$ we have 
\begin{equation}\label{160319ls}
c \|f\vert \cM_{\frac{u}{p},\Phi_{p,q}}(\rd)\|^* \le  \sup_{j\geq 1}j^{-1/q}\, \|f\vert \cM_{v(j),p(j)}(\rd)\| \le C  \|f\vert \cM_{\frac{u}{p},\Phi_{p,q}}(\rd)\|. 
\end{equation}

Consider a dyadic cube $Q$ and $\lambda>0$.
Using the Taylor expansion of $\Phi_{p,q}$ and the Stirling's formula, we have
\begin{align}
\int_Q\Phi_{p,q} \Bigl( \frac{|f(x)|}{\lambda}\Bigr)\dint x & =\sum_{j=0}^{\infty} \frac{1}{j !} \int_Q \frac{|f(x)|^{p+jq}}{\lambda^{p+jq}} \dint x \nonumber \\
&= \sum_{j=0}^{\infty} j^{-j} e^j (2\pi j)^{-1/2} \lambda^{-(p+jq)}\int_Q |f(x)|^{p+jq} \dint x . \nonumber
\end{align}
Let $\kappa \in\real$ (be at our disposal) and let
$$
\lambda_j:=(2\pi)^{\frac{1}{2(p+qj)}}  e^{-\frac{j}{p+qj}} j^{(\kappa-\frac{p}{q}+\frac{1}{2})/(p+qj)}, \quad j\in\nn .
$$
It can be easily seen that  the sequence $(\lambda_j)_j$ converges to $e^{-1/q}$, thus there are positive constants $c_0,c_1$ such that $0<c_0<\lambda_j<c_1<\infty$  for any $j\in \nn$. Therefore,
\begin{equation} \label{equiv}
\int_Q\Phi_{p,q}\Bigl( \frac{|f(x)|}{\lambda}\Bigr)\dint x \sim  \sum_{j=0}^{\infty} j^{\kappa-\frac{p}{q}-j}  \lambda^{-(p+jq)}\int_Q |f(x)|^{p+jq} \dint x .
\end{equation}
\smallskip

\emph{Step 2}.  We prove the left-hand side inequality in \eqref{160319ls}. Assume that
$$
\sup_{j\geq 1}j^{-1/q}\, \|f\vert \cM_{v(j),p(j)}(\rd)\|\leq \lambda.
$$
Then, for any dyadic cube $Q$ and any $j\in\nn$, it holds
$$
\lambda^{-(p+qj)} j^{-\frac{p+qj}{q}} |Q|^{\frac{p}{u}-1 }\int_Q |f(x)|^{p(j)} \dint x \leq 1.
$$
For this $\lambda$, inserting the above inequality in \eqref{equiv} entails
$$
\int_Q\Phi_{p,q}\Bigl( \frac{|f(x)|}{\lambda}\Bigr)\dint x \leq c \sum_{j=0}^{\infty} j^{\kappa} |Q|^{1-\frac{p}{u}}.
$$
By choosing $\kappa<-1$, we conclude that 
$$
\sup_{Q\in {\cal Q}}   |Q|^{\frac{p}{u}-1} \int_Q\Phi_{p,q}\Bigl( \frac{|f(x)|}{\lambda}\Bigr)\dint x \leq c.
$$

\emph{Step 3}.  It remains to show the right-hand side inequality in  \eqref{160319ls}. Let now
\[
\|f\vert \cM_{r,\Phi}(\rd)\|:= \sup_{Q\in\mathcal{Q}}\|f\|_{(r,\Phi);Q} \le 1
\]
and let $\varepsilon>0$. Then for any dyadic cube $Q$ there is $\lambda_Q$, $0<\lambda_Q\le \|f\|_{(r,\Phi);Q}+\varepsilon$,  such that 
$$
|Q|^{\frac{p}{u}-1} \int_Q\Phi_{p,q}\Bigl( \frac{|f(x)|}{\lambda_Q}\Bigr)\dint x \leq 1 .
$$
Then, by \eqref{equiv}, 
$$
 \sum_{j=0}^{\infty} j^{\kappa-\frac{p}{q}-j}  \lambda_Q^{-(p+jq)}  |Q|^{\frac{p(j)}{v(j)}-1} \int_Q |f(x)|^{p+jq} \dint x \leq c
$$ 
for all dyadic cubes $Q$ and for some positive constant c independent of $Q$. Hence, for any $j\in\nn$ and any dyadic cube $Q$, it holds
$$
j^{-\frac{p}{q}-j}    |Q|^{\frac{p(j)}{v(j)}-1} \int_Q |f(x)|^{p+jq} \dint x \leq c j^{-\kappa} \lambda_Q^{p+jq}\le c j^{-\kappa} (\|f\|_{(r,\Phi);Q}+\varepsilon )^{p+jq}.
$$
Taking the infimum over $\varepsilon$ we get 
$$
j^{-\frac{p}{q}-j}    |Q|^{\frac{p(j)}{v(j)}-1} \int_Q |f(x)|^{p+jq} \dint x \leq  c j^{-\kappa} \|f\|_{(r,\Phi);Q}^{p+jq}
$$
and afterwards taking the supremum over all dyadic cubes gives 
$$
j^{-1/q}  \| f\vert \cM_{v(j),p(j)}(\rd)\| \leq c^{1/p(j)} j^{-\kappa/p(j)} \|f\vert \cM_{r,\Phi}(\rd)\|.
$$
Now we take the supremum over all $j$. The expression on the right-hand side is of the size $j^{1/j}$ so it can be  bounded by a positive constant. This   yields 
$$\sup_{j\geq 1} j^{-1/q}  \| f\vert \cM_{v(j),p(j)}(\rd)\| \leq c \|f\vert \cM_{r,\Phi}(\rd)\|.
$$
\end{proof}


\begin{theorem}\label{emb-N-in-MO}
Let $0< p\le u<\infty$ and  $1<q \leq  \infty$. Then 
$$
\cN^{\frac{d}{u}}_{u,p,q}(\rd) \hookrightarrow \cM_{\frac{u}{p},\Phi_{p,q'}}(\rd),
$$
where  $q'$ is the conjugate exponent of $q$.
\end{theorem}

\begin{proof}
For each $j\in\nn$, by Lemma \ref{lemma1} with $r=p(j)=p+qj$ and $v(j)=\frac{u}{p}p(j)$, 
$$
\|f|\cM_{v(j),p(j)}(\rd)\|  \le c \{v(j)\}^{1/q'} \|f|\MBd(\rd)\|,
$$
where $c$ is a positive constant independent of $v(j)$ and $p(j)$, and thus of $j$.
Since $\{v(j)\}^{1/q'} \sim j^{1/q'}$, using also Lemma \ref{lemma-equiv-OM}, we get
$$
\|f\vert \cM_{\frac{u}{p},\Phi_{p,q'}}(\rd)\| \sim \sup_{j\geq 1}j^{-1/q'}\, \|f\vert \cM_{v(j),p(j)}(\rd)\| \leq c  \|f|\MBd(\rd)\|.
$$
\end{proof}

\begin{corollary} \label{emb-E-in-MO}
Let $0< p\le u<\infty$ and $0<q\leq \infty$. Then 
$$
\cE^{\frac{d}{u}}_{u,p,q}(\rd) \hookrightarrow \cM_{\frac{u}{p},\Phi_{p,1}}(\rd).
$$
\end{corollary}
\begin{proof}
The result follows from the above theorem and elementary embeddings:
$$
\cE^{\frac{d}{u}}_{u,p,\infty}(\rd) \hookrightarrow \cN^{\frac{d}{u}}_{u,p,\infty}(\rd) \hookrightarrow \cM_{\frac{u}{p},\Phi_{p,1}}(\rd).
$$
\end{proof}

\begin{remark} According to   \cite[Cor.~1.5]{sw} it holds
$$
\cE^{\frac{d}{u}}_{u,p,2}(\rd) \hookrightarrow \cM_{\frac{u}{p},\Phi_{p}}(\rd),
$$ 
where $1<p\leq u<\infty$ and 
$$
\Phi_p(t):=\sum_{j=j_p}^{\infty}\frac{t^j}{j!}, \quad t\geq 0,
$$
 with $j_p:=\min\{j\in\nn: j\geq p\}$. Since there exists a constant $c>0$ such that 
 $$
\Phi_{p,1}(t) \geq c \Phi_p(t) \quad \text{for all} \quad t\geq 0, 
 $$
 it turns out that 
 $$
\cM_{\frac{u}{p},\Phi_{p,1}}(\rd) \hookrightarrow \cM_{\frac{u}{p},\Phi_{p}}(\rd).
 $$
Hence Corollary \ref{emb-E-in-MO} does not only extend \cite[Cor.~1.5]{sw}  from $q=2$ to any $1<q\leq \infty$, it improves it.
\end{remark}

\subsection{Embeddings in generalised Morrey spaces}

Now we turn to the generalised Morrey spaces. The spaces  were extensively studied, cf. Nakai \cite{nakai2}, Nakamura, Noi and Sawano \cite{nns}, Sawano and Wadade \cite{sw} and the references given there.   Let $0 < r < \infty$ 
and  let $\varphi:[0,\infty)\rightarrow [0,\infty)$ be a suitable function. For a locally $r$-integrable function $f$ we put
\begin{equation}\label{def-gen-morrey}
\|f|\cM_r^{\varphi}\|  :=\, \sup_{Q\in\mathcal{Q}} \varphi(|Q|)
\biggl( |Q|^{-1}\int_{Q} |f(y)|^r \dint y \biggr)^{1/r}\, .
\end{equation}

The space $\cM_r^{\varphi}(\rd)$ is the set of all measurable functions $f$ for which the quasi-norm \eqref{def-gen-morrey} is finite. If $\varphi(t)=t^{1/u}$, $0<r\le u$, then  the definition coincides with the definition of the Morrey space $\cM_{u,r}(\rd)$.  
Since we will work with the given examples of functions $\varphi$ we avoid the discussions which functions $\varphi$ define the reasonable spaces, and we refer the interested reader to the above mentioned papers.  

We start with a proposition that somehow compares the Orlicz-Morrey spaces with generalised Morrey spaces we will use.  

\begin{proposition}\label{MO-gm}
Let $0< p\le u<\infty$,   $1\le q <  \infty$  and $r\ge p+q$. Then there is a positive constant $C$ depending on $u$, $p$,  $q$ and $r$ such that the following inequality
\[
\left(\frac{1}{|Q|}\int_Q|f(x)|^{r}\dint x\right)^{\frac{1}{r}}\leq c \,(1+|Q|)^{-\frac{p}{ru}} \left(\log \left(e+|Q|^{-1}\right)\right)^{\frac{1}{q}} 
\|f\vert \cM_{\frac{u}{p},\Phi_{p,q}}(\rd)\|  
\]
holds for all dyadic cubes $Q$ and all $f\in\cM_{\frac{u}{p},\Phi_{p,q}}(\rd)$.
\end{proposition}

\begin{proof}
  It should be clear that it is sufficient to consider the case $r=p+j_0q$ for some $j_0\in \nn$. Note that this refers to $r=p(j_0)$ in the notation of Lemma~\ref{lemma-equiv-OM}.
  
  Let $Q$ be a  dyadic cube with side length $2^{-j}$ and $|Q|=2^{-jd}$, $j\in\zz$.  
Assume first $j\leq 0$. Then by Lemma  \ref{lemma-equiv-OM} we get 
\begin{align}
\left(\frac{1}{|Q|} \int_Q|f(x)|^{r}\dint x\right)^{\frac{1}{r}} & \leq C \,j_0^{1/q} |Q|^{-\frac{p}{ur}} \|f\vert \cM_{\frac{u}{p},\Phi_{p,q}}(\rd)\| \nonumber \\
 & \leq C (1+|Q|)^{-\frac{p}{ur}} \|f\vert \cM_{\frac{u}{p},\Phi_{p,q}}(\rd)\| . \label{or-gm1}
\end{align}

Next we assume that $j\ge j_0$. Then H\"older's inequality and Lemma~\ref{lemma-equiv-OM} imply 
\begin{align}
\left(\frac{1}{|Q|} \int_Q|f(x)|^{r}\dint x\right)^{\frac{1}{r}} & \leq  \left(\frac{1}{|Q|} \int_Q|f(x)|^{p(j)}\dint x\right)^{\frac{1}{p(j)}} \nonumber \\ 
 & \le
C \,j^{1/q} |Q|^{-\frac{p}{u p(j)}} \|f\vert \cM_{\frac{u}{p},\Phi_{p,q}}(\rd)\| \nonumber \\ 
& \leq C \left(\log(e+|Q|^{-1})\right)^{\frac{1}{q}}  \|f\vert \cM_{\frac{u}{p},\Phi_{p,q}}(\rd)\| \label{or-gm2}
\end{align}
since 
\[
 \log(e+|Q|^{-1}) \sim j \quad \text{and}\quad  |Q|^{-\frac{p}{u p(j)}}= 2^{\frac{jdp}{u (p+ jq)}} \le 2^{\frac{dp}{u q}} . 
\]

At the end we consider the cubes with  $0<j < j_0$. We have
\begin{align*}
 2^{-\frac{j_0 d}{v(j_0)}} & \sup_{Q:\; |Q|=2^{-jd},\;0<j < j_0}\left(\frac{1}{|Q|} \int_Q|f(x)|^{r}\dint x\right)^{\frac{1}{r}} \\ 
&\leq \ \sup_{Q} 
|Q|^{\frac{1}{v(j_0)}} \left(\frac{1}{|Q|} \int_Q|f(x)|^{r}\dint x\right)^{\frac{1}{r}} 
\le C \,j_0^{\frac1q}\  \|f\vert \cM_{\frac{u}{p},\Phi_{p,q}}(\rd)\|,
\end{align*}
by Lemma~\ref{lemma-equiv-OM}, recall $r=p(j_0)$. Thus
\begin{equation}\label{or-gm3}
\left(\frac{1}{|Q|} \int_Q|f(x)|^{r}\dint x\right)^{\frac{1}{r}}  \leq C \left( \log(e+|Q|^{-1})\right)^{\frac{1}{q}}  \|f\vert \cM_{\frac{u}{p},\Phi_{p,q}}(\rd)\| .
\end{equation}
Consequently, the inequalities \eqref{or-gm1}-\eqref{or-gm3} prove the proposition.
\end{proof}

The next corollary is an immediate consequence of Theorem~\ref{emb-N-in-MO},  Corollary~\ref{emb-E-in-MO} and Proposition~\ref{MO-gm}. 

\begin{corollary}
Let $0< p\le u<\infty$ and  $0< q \leq  \infty$.
\begin{itemize}
\item[{\upshape\bfseries (i)}]
If  $1< q \leq  \infty$   and $r\ge p+q'$ . Then there is a positive constant $C$ 
such that the  inequality   
\begin{equation}\label{gm-N1} 
\left(\frac{1}{|Q|}\int_Q|f(x)|^{r}\dint x\right)^{\frac{1}{r}}\leq c \,(1+|Q|)^{-\frac{p}{ru}} \left(\log \left(e+|Q|^{-1}\right)\right)^{\frac{1}{q'}} \|f|\MBd(\real^d)\|
\end{equation}
holds for all dyadic cubes $Q$ and all $f\in\MBd(\rd)$.
\item[{\upshape\bfseries (ii)}]
If  $ 0 < q \leq  \infty$   and $r\ge p + 1$ . Then there is a positive constant $C$  
such that the  inequality   
\begin{equation}\label{gm-E1}
\left(\frac{1}{|Q|}\int_Q|f(x)|^{r}\dint x\right)^{\frac{1}{r}}\leq c \,(1+|Q|)^{-\frac{p}{ru}} \left(\log \left(e+|Q|^{-1}\right)\right) \|f|\MFd(\real^d)\|
\end{equation}
holds for all dyadic cubes $Q$ and all $f\in\MFd(\rd)$. 
\end{itemize}
\end{corollary}

The inequalities \eqref{gm-N1} and \eqref{gm-E1} can be extended to the smallest values of $r$ and $q$.  
Recall that $q'$ is defined by $q'=\frac{q}{q-1}$ if $1<q<\infty$ and $q'=\infty$ if $0<q\le 1$, where the usual convention $1/\infty=0$ is assumed.

\begin{theorem} \label{theorem-MB}
Let $0< p\le u<\infty$ and $0<q\le \infty$. If $1\leq r<\infty$, then  there exists a positive constant $c$ such that 
\[
\left(\frac{1}{|Q|}\int_Q|f(x)|^{r}\dint x\right)^{\frac{1}{r}}\leq c \,(1+|Q|)^{-\frac{\min(1,\frac{p}{r})}{u  }}   \left(\log \left(e+|Q|^{-1}\right)\right)^{\frac{1}{q'}} \|f|\MBd(\real^d)\|
\]
holds for all dyadic cubes $Q$ and all $f\in\MBd(\rd)$.  
\end{theorem}

\begin{proof}
\emph{Step 1}. Given $r\geq 1$, let $r_0$ be such that $r_0>\max(p,r)$.  By Lemma \ref{lemma1},  there exists a positive constant $c$, not depending on $r_0$ and $v_0$, with $v_0=\frac{ur_0}{p}$, such that 
the inequality
\begin{equation}\label{main-ineq}
|Q|^{\frac{1}{v_0}-\frac{1}{r_0}}\left(\int_Q|f(x)|^{r_0}\dint x\right)^{\frac{1}{r_0}}\leq c\, v_0^{\frac{1}{q'}}  \|f|\MBd(\real^d)\|
\end{equation}
holds for all dyadic cubes $Q$ and all $f\in\MBd(\real^d)$.
The  H\"older inequality and \eqref{main-ineq} yield
 \begin{align} 
\left(\int_Q|f(x)|^{r}\dint x\right)^{\frac{1}{r}} &\leq |Q|^{\frac{1}{r}-\frac{1}{r_0}} \left(\int_Q|f(x)|^{r_0}\dint x\right)^{\frac{1}{r_0}}  \nonumber \\
&\leq  c\, |Q|^{\frac{1}{r}-\frac{1}{v_0}} \, v_0^{\frac{1}{q'}}  \,   \|f|\MBd(\real^d)\|  \label{thmMB:a}
\end{align}
for all dyadic cubes $Q$ and all $f\in\MBd(\real^d)$.

For convenience we deal with the case $q\geq 1$, the other case is even easier.  
Assume first that the cubes are small, that is, they satisfy $|Q|< e^{-u\max(1,\frac{r}{p})}$. Then $v_0=\log (|Q|^{-1})$ and $r_0=\frac{p}{u} v_0$ satisfy the above assumptions. Hence \eqref{thmMB:a} leads to
\begin{equation}\label{thmMB:b}
\left(\frac{1}{|Q|}\int_Q|f(x)|^{r}\dint x\right)^{\frac{1}{r}}\leq c\,   \left(\log \left(e+ |Q|^{-1}\right)\right)^{\frac{1}{q'}} \|f|\MBd(\real^d)\|,
\end{equation}
for any small enough cube $Q$ with $|Q|< e^{-u\max(1,\frac{r}{p})}$, and all $f\in\MBd(\real^d)$. It remains to deal with the bigger cubes.\\

\emph{Step 2}.  Let $r\geq p$. Elementary embeddings and \cite[Thm.~3.3]{hs12} yield
\[
\MBd(\rd) \hookrightarrow  \cN^0_{\frac{ur}{p},r,1} (\rd) \hookrightarrow  \cE^0_{\frac{ur}{p},r,2} (\rd) =\cM_{\frac{ur}{p},r} (\rd), \quad \text{if}\quad r>1,
\]
and 
\[
\MBd(\rd) \hookrightarrow  \cN^0_{\frac{u}{p},1,1} (\rd) \hookrightarrow \cM_{\frac{u}{p},1} (\rd), \quad \text{if}\quad r=1,
\]
where the last embedding is due to Theorem~\ref{theorem-n_290319}.
Therefore, for $r\geq 1$  and $r\geq p$, we have 
\begin{equation}\label{thmMB:c}
\left(\frac{1}{|Q|}\int_Q|f(x)|^{r}\dint x\right)^{\frac{1}{r}}\leq c\,  |Q|^{-\frac{p}{ur}} \|f|\MBd(\real^d)\|,
\end{equation}
for all dyadic cubes $Q$ and all $f\in\MBd(\real^d)$. Together with Step~1 this completes the argument in case of $r\geq p$.\\

\emph{Step 3}.   Let $1\leq r<p$. Using H\"older's inequality we obtain
\begin{align}
\left(\int_Q|f(x)|^{r}\dint x\right)^{\frac{1}{r}} & \leq |Q|^{\frac{1}{r}-\frac{1}{p}} \left(\int_Q|f(x)|^{p}\dint x\right)^{\frac{1}{p}}
 \leq |Q|^{\frac{1}{r}-\frac{1}{u}}    \|f|\cM_{u,p}(\real^d)\|   \nonumber \\
& \leq c  |Q|^{\frac{1}{r}-\frac{1}{u}}    \|f|\MBd(\real^d)\|,   \label{thmMB:d}
\end{align}
for all dyadic cubes $Q$ and all $f\in\MBd(\real^d)$, where in the last step we used the fact that 
\[
 \cN^{\frac{d}{u}}_{u,p,\infty} (\rd) \hookrightarrow  \cN^0_{u,p,1} (\rd) \hookrightarrow  \cE^0_{u,p,2} (\rd) =\cM_{u,p} (\rd).
 \]
Again the final outcome in this case follows from Step~1 and  \eqref{thmMB:d}.
\end{proof}

\begin{remark} \label{rmk-MB}
(i) In terms of embeddings in generalised Morrey spaces,  what has been proved could be stated as follows. Let $0< p\le u<\infty$,  $0<q\le \infty$, and $1\leq r<\infty$.  Then 
\[
\cN^{\frac{d}{u}}_{u,p,q} (\rd) \hookrightarrow \cM_r^{\varphi_{r,q}}(\rd)
\]
where
\begin{equation} \label{phi-r}
\varphi_{r,q}(t)=
\begin{cases}
\big(\log(t^{-1})\big)^{-\frac{1}{q'}} & \text{if} \quad 0<t< e^{-\frac{u}{d}\max(1,\frac{r}{p})}, \\
t^{\frac{d}{u}\min(1,\frac{p}{r})} & \text{if} \quad t \geq  e^{-\frac{u}{d}\max(1,\frac{r}{p})} .
\end{cases}
\end{equation}
(ii) If we would consider local generalised Morrey spaces $L\cM_p^{\varphi}(\rd)$, where the supremum taken in the definition of the norm is restricted to cubes with volume less or equal than 1 (cf. \cite[page 7]{t14}), then we can state that
\[
\cN^{\frac{d}{u}}_{u,p,q} (\rd) \hookrightarrow L\cM_r^{\varphi}(\rd), \qquad \varphi(t)=|\log(t)|^{-\frac{1}{q'}}, \quad 1\leq r<\infty.
\]
\end{remark}

\begin{corollary}
Let $0< p\le u<\infty$ and $0<q\le \infty$. If $1\leq r<\infty$, then  there exists a positive constant $c$ such that 
\[
\left(\frac{1}{|Q|}\int_Q|f(x)|^{r}\dint x\right)^{\frac{1}{r}}\leq c \,(1+|Q|)^{-\frac{\min(1,\frac{p}{r})}{u  }}   \log \left(e+|Q|^{-1}\right) \|f|\MFd(\real^d)\|
\]
holds for all dyadic cubes $Q$ and all $f\in\MFd(\real^d)$. 
\end{corollary}

\begin{proof} The outcome is a direct consequence of Theorem \ref{theorem-MB} taking into account the embedding
\[
\MFd (\rd) \hookrightarrow \cN^{\frac{d}{u}}_{u,p,\infty} (\rd).
\]

\end{proof}

\begin{remark} (i) As in Remark \ref{rmk-MB}, we can state the following: Let $0< p\le u<\infty$,  $0<q\le \infty$, and $1\leq r<\infty$.  Then 
\[
\cE^{\frac{d}{u}}_{u,p,q} (\rd) \hookrightarrow \cM_r^{\varphi_r}(\rd) \qquad \text{and} \qquad \cE^{\frac{d}{u}}_{u,p,q} (\rd) \hookrightarrow L\cM_r^{\varphi}(\rd)
\]
with $\varphi_r=\varphi_{r,\infty}$ given by \eqref{phi-r} and $\varphi(t)=|\log(t)|^{-1}$.

(ii) In the particular case of $p>1$, $q=2$ and $r=1$, the result in the above corollary coincides with Theorem~5.1 of \cite{sw}.

(iii) In the particular case of $p=1$, $q=2$ and $r=1$, the result in the above corollary is comparable with Proposition~1.7 of \cite{EGNS14}.

\end{remark}

\section{Further embeddings into spaces with smoothness $s=0$}

In the preceding subsection we dealt with (limiting) embeddings of spaces $\MA$ when $s=\frac{d}{u}$, into spaces of Orlicz-Morrey type or generalised Morrey type. In Corollary~\ref{theorem-2-n} the parallel setting was studied for embeddings into Morrey spaces $\mathcal{M}_{v,q}(\rd)$ when $s=\frac{d}{u}-\frac{d}{v}$. For convenience we briefly recall the forerunners, that is, when $\MA(\rd)$ is embedded into classical Lebesgue spaces $L_r(\rd)$, and into the space $C(\rd)$ of bounded, uniformly continuous functions. We also consider the target space $\bmo(\real^d)$, i.e.,
the local (non-homogeneous) space of functions of bounded mean oscillation, consisting of all locally integrable
functions $\ f\in \Lloc(\real^d) $ satisfying that
\begin{equation*}
 \left\| f \right\|_{\bmo(\rd)}:=
\sup_{|Q|\leq 1}\; \frac{1}{|Q|} \int\limits_Q |f(x)-f_Q| \dint x + \sup_{|Q|>
1}\; \frac{1}{|Q|} \int\limits_Q |f(x)| \dint x<\fz,
\end{equation*}
where $ Q $ appearing in the above definition runs over all cubes in $\real^d$, and $ f_Q $ denotes the mean value of $ f $ with
respect to $ Q$, namely, $ f_Q := \frac{1}{|Q|} \;\int_Q f(x)\dint x$.

Most of these results have been obtained in different papers before, we recall it for completeness and to simplify the comparison with our new findings presented above. 
\begin{corollary}\label{embtobottomline}
  Let $0 <p\leq u<\infty$, $q\in(0,\infty]$ and $s\in \rr$.
    \begin{itemize}
\item[{\upshape\bfseries (i)}] Then 
\begin{equation*}
\MB(\rd) \hookrightarrow C(\rd)\quad \iff\quad 
\begin{cases} s>\frac{d}{u}, & \text{or}\\ s=\frac{d}{u}& \text{and}\quad 
 q\leq 1, 
\end{cases}
\end{equation*}
and
\begin{equation*}
  \MF(\rd) \hookrightarrow C(\rd)\quad \iff\quad \begin{cases} s>\frac{d}{u}, & \text{or}\\ s=\frac{d}{u}&
\text{and}\quad u=p\leq 1.
\end{cases}
\end{equation*}
Here $C(\rd)$ can be replaced by $L_\infty(\rd)$.
\item[{\upshape\bfseries (ii)}]
Then 
  \begin{equation}\label{bmo1}
\MA(\rd)\hookrightarrow \bmo(\rd)\qquad\text{if, and only if,}\qquad s\geq\frac{d}{u}.
\end{equation}
\item[{\upshape\bfseries (iii)}]
  Let $1\leq r<\infty$. If $p<u$, then $\MA(\rd)$ is never embedded into $L_r(\rd)$. If $p=u$, then
  \begin{equation*}
    \mathcal{N}^s_{u,u,q}(\rd)\hookrightarrow L_r(\rd) \iff\ r\geq u\quad\text{and}\quad \begin{cases}
s>\frac{d}{u}-\frac{d}{r},& \text{or}\\ s=\frac{d}{u}-\frac{d}{r}& \text{and}\quad q\leq r,  
\end{cases}
\end{equation*}
and
\begin{equation*}
  \mathcal{E}^s_{u,u,q}(\rd) \hookrightarrow L_r(\rd)  \iff\ r\geq u\quad\text{and}\quad \begin{cases}
 s\geq \frac{d}{u}-\frac{d}{r}& \text{and}\quad s>0,\ \text{or}\\
s=\frac{d}{u}-\frac{d}{r}=0 & \text{and}\quad q\leq 2.
\end{cases}
\end{equation*}
    \end{itemize}
  \end{corollary}

\begin{proof}
  {\em Step 1}.~ Part (i) is well-known, we refer to \cite[Prop.~5.5]{hs12} and \cite[Prop.~3.8]{hs13} for the Morrey situation when $p<u$, while the classical setting $p=u$ can be found in  \cite[Theorem~11.4]{T-func}.
  
  {\em Step 2}.~ We prove (ii). In case of $\MA=\MF$, this follows from the analogous statement for $F^{s,\tau}_{p,q}(\real^d)$ spaces and the coincidence $\MF(\real^d)=F^{s,\tau}_{p,q}(\real^d)$ if $\tau=\frac{1}{p}-\frac{1}{u}$, cf. \cite[Props.~5.13, 5.14]{HSYY} and \eqref{Ft-E}. 
The similar statement for $\MB(\real^d)$ spaces 
\begin{equation}\label{bmo2}
\MB(\real^d)\hookrightarrow \bmo(\real^d)\qquad\text{if, and only if,}\qquad s\geq\frac{d}{u}
\end{equation}
can be proved analogously. Let $\tau=\frac{1}{p}-\frac{1}{u}$. Then, in view of \eqref{Bt-N} and the subsequent remark, 
\[
\mathcal{N}^s_{u,p,\infty}(\real^d) = B^{s,\tau}_{p,\infty}(\real^d)\hookrightarrow \bmo(\real^d)= B^{0,\frac{1}{2}}_{2,2}(\real^d)
\]
if $\frac{1}{p}-\frac{1}{u}- \frac{1}{2}\not= 0$, cf. \cite[Prop.~5.10, Theorem 2.5]{HSYY}. If $\frac{1}{p}-\frac{1}{u}- \frac{1}{2} = 0$,
then we can choose $r$ such that $u<r<p$ and  $\frac{1}{r}-\frac{1}{u}- \frac{1}{2} > 0$. Hence
\[
\mathcal{N}^s_{u,p,\infty}(\real^d) \hookrightarrow \mathcal{N}^s_{u,r,\infty}(\real^d) \hookrightarrow \bmo(\real^d). 
\]
 This proves the sufficiency of the conditions. 
 
 To prove necessity let us take  $s<\frac{d}{u}$.  If $\mathcal{N}^s_{u,p,\infty}(\real^d) \hookrightarrow \bmo(\real^d)$, then $\MF(\real^d)\hookrightarrow \bmo(\real^d)$. This contradicts \eqref{bmo1}.  

 {\em Step 3}.~Part (iii) in case of $p<u$ can be found in \cite{hs12,hs13}, whereas the classical results for $p=u$ are well-known.
\end{proof}

 \begin{remark}
A partial forerunner of (i) can be found in \cite[Prop.~1.11]{Saw2010} dealing with the sufficiency part; see also \cite{s011a}.
In some sense the embeddings into $C(\rd)$ and $\bmo(\rd)$ can be understood as limiting cases of Corollary~\ref{theorem-2-n} when $v\to\infty$, whereas the embedding into $L_r(\rd)$ refers to the situation of $r=v=q$ in Corollary~\ref{theorem-2-n}.

 \end{remark}

\begin{remark}
  Based on arguments on the known properties of Triebel-Lizorkin type spaces one can easily strengthen Remark~\ref{at-into-M}, in particular \eqref{Ft-M} with \eqref{At-M-ass}, as follows. Let $0<p<\infty$, $0<q\leq\infty$, $1\leq r\leq v<\infty$, $s\in\real$, $\tau\geq 0$. Then
  \begin{equation}\label{Ft-M-gen}
  \ft(\rd)\hookrightarrow \cM_{v,r}(\rd)
  \end{equation}
  if, and only if,
  \[
r\leq v(1-p\tau),\quad\text{and}\quad 
  \begin{cases}
s>d(\frac1p-\tau-\frac1v)\geq 0, &\text{or}\\ s=d(\frac1p-\tau-\frac1v)>0, &\text{or}\\ 
s=d(\frac1p-\tau-\frac1v)=0,&\text{and}\quad q\leq 2.
  \end{cases}
  \]
  The case $0\leq \tau<\frac1p$ is covered by Remark~\ref{at-into-M} together with the usual monotonicity  for spaces $\ft(\rd)$ and $\cM_{v,r}(\rd)$, see also the forerunner \cite[Thm.~3.1]{hs13}. Thus it remains to disprove any embedding of type \eqref{Ft-M-gen} whenever $\tau\geq \frac1p$. Assume first $\tau > \frac{1}{p}$ or  $\tau = \frac{1}{p}$ and $q=\infty$. Thus the coincidence $\ft(\rd)=B^{s+d(\tau-\frac{1}{p})}_{\infty,\infty}(\rd)$, cf. \cite{yy02}, together with \eqref{Ft-M-gen} would imply  
\[ B^{s+d(\tau-\frac{1}{p})}_{\infty,\infty}(\rd)=\ft(\rd) \hookrightarrow \cM_{v,r}(\rd) \hookrightarrow\cN^0_{v,r,\infty}(\rd), \]
and hence $v=\infty$ in view of \cite[Thm.~3.3]{hs12}. But this contradicts our general assumption. Otherwise, if $\tau = \frac{1}{p}$ and $q<\infty$, then we may use \cite[Prop.~2.6]{ysy} which states, that $\ft(\rd) \hookrightarrow B^{s+d(\tau-\frac1p)}_{\infty,\infty}(\rd)$. We choose a number $\varrho$ such that $v<\varrho<\infty$, and apply an embedding proved by Marschall, cf. \cite{Mar}, to obtain
\[ B^{s+\frac{d}{\varrho}}_{\varrho,\infty}(\rd)\hookrightarrow F^s_{\infty,q}(\rd)=F^{s,\frac1p}_{p,q}(\rd)\hookrightarrow \cM_{v,r}(\rd) \hookrightarrow\cN^0_{v,r,\infty}(\rd). \]
Again the embedding \eqref{Ft-M-gen} would thus lead to $\varrho\leq v$ in view of \cite[Thm.~3.3]{hs12}, i.e., to a contradiction by our choice of $\varrho$. Here we also used the identification $F^s_{\infty,q}(\rd)=F^{s,\frac1p}_{p,q}(\rd)$, see \cite[Props.~3.4 and 3.5]{s011} and \cite[Rem.~10]{s011a}. In a completely parallel way one can show that an embedding $ \bt(\rd)\hookrightarrow \cM_{v,r}(\rd)$ is never possible when $\tau\geq \frac1p$.

Note that the limiting case $v=\infty$ is covered by \cite[Prop.~5.4]{HSYY}, \cite[Prop.~4.1]{YHMSY}, in view of $\cM_{\infty,r}(\rd)=L_\infty(\rd)$. In particular, if $0<p<\infty$, $0<q\leq\infty$, $s\in\real$ and $\tau> 0$, then
\[
\ft(\rd)\hookrightarrow L_\infty(\rd) \quad\text{if, and only if,}\quad s>d\left(\frac1p-\tau\right).
\]
The result for $\bt(\rd)$ looks alike.
\end{remark}



\bigskip

\noindent Dorothee D. Haroske\\
\noindent  Institute of Mathematics,
Friedrich Schiller University Jena, 07737 Jena, Germany\\
\noindent {\it E-mail}:  \texttt{dorothee.haroske@uni-jena.de}

\bigskip

\noindent Susana D. Moura (corresponding author)\\
\noindent  CMUC, Department of Mathematics, University of Coimbra,  EC Santa Cruz, 3001-501 Coimbra, Portugal\\
\noindent {\it E-mail}:  \texttt{smpsd@mat.uc.pt}

\bigskip

\noindent Leszek Skrzypczak\\
\noindent  Faculty of Mathematics and Computer Science,
Adam Mickiewicz University, Ul. Umultowska 87, 61-614 Pozna\'n,
Poland\\
\noindent {\it E-mail}:  \texttt{lskrzyp@amu.edu.pl}


\begin{thebibliography}{99}


%
%
%
%
%
%
%
%
%
%
%
%
%
%


\bibitem{EGNS14}
H.G. Eridani , H. Gunawan, E. Nakai,  and Y. Sawano, 
Characterizations for the generalized fractional integral operators on {M}orrey spaces, 
 Math. inequal. Appl. 17 (2014), no. 2, 761-777. 

%
%
%
%
%




%
%
%

\bibitem{HaSM3}
D.D. Haroske and S.D. Moura,
Some specific unboundedness property in Smoothness Morrey Spaces. 
The non-existence of growth envelopes in the subcritical case, 
Acta Math. Sin. (Engl. Ser.) 32 (2016) 137-152.


\bibitem{hms}
D.D. Haroske, S.D. Moura, L. Skrzypczak,  
Smoothness Morrey Spaces of regular distributions, and some unboundedness properties, 
 Nonlinear Analysis Series A: Theory, Methods and  Applications  139 (2016), 218-244.

\bibitem{hs12}
D.D. Haroske and L. Skrzypczak,
Continuous embeddings of Besov-Morrey function spaces,
Acta Math. Sin. (Engl. Ser.) 28 (2012), 1307-1328.





%
%
\bibitem{hs13}
D.D. Haroske and L.~Skrzypczak,
{On Sobolev and Franke-Jawerth embeddings of smoothness Morrey
  spaces}, {Rev. Mat. Complut.} 27 (2014), 541-573.
  

\bibitem{ho}  
K.-P. Ho, Atomic decompositions of weighted Hardy-Morrey spaces,
Hokkaido Math. J. 42 (2013), 131-157.


\bibitem{KY}
H.~Kozono and M.~Yamazaki,
Semilinear heat equations and the {N}avier-{S}tokes equation with distributions in new function spaces as initial data,
Comm. Partial Differential Equations 19 (1994), 959-1014.
  
 \bibitem{kufner} 
 A. Kufner, O. John, and S. Fu\v{c}{\'\i}k,
 Function Spaces,
Function Spaces, Noordhoff International Publishing, Leyden, 1977.

\bibitem{Mar}
J.~Marschall, On the boundedness and compactness of nonregular pseudo-differential operators, 
Math. Nachr. 175 (1995), 231-262. 

\bibitem{Maz}
A.L.~Mazzucato,
Besov-{M}orrey spaces: function space theory and applications to non-linear {PDE},
Trans. Amer. Math. Soc. 355 (2003), 1297-1364.

 \bibitem{MV} 
 R. Meise and  D. Vogt,
 Introduction to functional analysis,
Clarendon Press, Oxford, 2004.

\bibitem{Mor}
C.B.~Morrey,
On the solutions of quasi-linear elliptic partial differential equations,
Trans. Amer. Math. Soc. 43 (1938), 126-166.

\bibitem{nakai}
E. Nakai, 
Generalized fractional integrals on Orlicz-Morrey spaces,
Banach and Function Spaces, Yokohama, pp. 323-333, 2004.

\bibitem{nakai2}
E. Nakai, Hardy-Littlewood maximal operator, singular integral operators and the Riesz potentials
on generalized Morrey spaces. Math. Nachr. 166(1994) , 95-103.

\bibitem{nns}
S.~Nakamura, T.~Noi, Y.~Sawano,  
Generalized Morrey spaces and trace operator, 
Sci. China Math. 59 (2016), 281-336.

\bibitem{Pee}
J.~Peetre,
On the theory of {${\cal L}_{p,\lambda }$} spaces,
J. Funct. Anal. 4 (1969), 71-87.

\bibitem{MR-1}
M.~Rosenthal,
Local means, wavelet bases, representations, and isomorphisms in Besov-Morrey and Triebel-Lizorkin-Morrey spaces,
Math. Nachr. 286 (2013), 59-87.

\bibitem{Saw2}
Y.~Sawano,
Wavelet characterizations of {B}esov-{M}orrey and  {T}riebel-{L}izorkin-{M}orrey spaces,
Funct. Approx. Comment. Math. 38 (2008), 93-107.

\bibitem{Saw1}
Y.~Sawano,
A note on {B}esov-{M}orrey spaces and {T}riebel-{L}izorkin-{M}orrey spaces,
Acta Math. Sin. (Engl. Ser.) 25 (2009), 1223-1242.

\bibitem{ST2}
Y.~Sawano and H.~Tanaka,
Decompositions of {B}esov-{M}orrey spaces and  {T}riebel-{L}izorkin-{M}orrey spaces,
Math. Z. 257 (2007), 871-905.

\bibitem{ST1}
Y.~Sawano and H.~Tanaka,
Besov-{M}orrey spaces and {T}riebel-{L}izorkin-{M}orrey spaces for non-doubling measures,
Math. Nachr. 282 (2009), 1788-1810.

\bibitem{sw}
Y.~Sawano and H. Wadade,
On Gagliardo-Nirenberg type inequality in the critical Sobolev-Morrey space, 
J. Fourier Anal. Appl. 19 (2013), 20-47.

%
%

\bibitem{Saw2010}
Y.~Sawano,
Besov-Morrey spaces and Triebel-Lizorkin-Morrey spaces on domains,
Math. Nachr. 283 (2010), no.~10,  1456-1487. 

%
\bibitem{s011}
W.~Sickel, Smoothness spaces related to Morrey spaces -- a survey. I, Eurasian Math. J. 3 (2012), 110-149.

\bibitem{s011a}
W.~Sickel, Smoothness spaces related to Morrey spaces -- a survey. II, Eurasian Math. J. 4 (2013), 82-124.

\bibitem{ST95}
W.~Sickel and H.~Triebel,
H\"older inequalities and sharp embeddings in function spaces of {$B^s_{p,q}$} and {$F^s_{p,q}$} type,
Z. Anal. Anwendungen 14 (1995), 105-140.

\bibitem{Str}
R.S.~Strichartz, 
 A note on Trudinger's extension of 
Sobolev's inequatlity. 
Indiana Univ.\ Math.\ J. 58(1972), 841-842.

\bibitem{TX}
L.~Tang and J.~Xu,
Some properties of {M}orrey type {B}esov-{T}riebel spaces,
Math. Nachr. 278 (2005), 904-917.
 
\bibitem{Tri93} 
{H.~Triebel}, {Approximation numbers and entropy
numbers of embeddings of fractional Besov-Sobolev spaces in Orlicz 
spaces.} Proc.\ London Math.\ Soc.\ 66(1993), 589-618.
 
 \bibitem{T-F1}
H.~Triebel,  Theory of Function Spaces,
Birkh\"auser, Basel, 1983.
%

\bibitem{T-func}
H.~Triebel,
The Structure of Functions,
Birkh\"auser, Basel, 2001.


\bibitem{t13} H. Triebel, Local Function Spaces,
Heat and Navier-Stokes Equations, EMS Tracts in Mathematics 20,
European Mathematical Society (EMS), Z\"urich, 2013.

\bibitem{t14} H. Triebel, Hybrid Function Spaces,
Heat and Navier-Stokes Equations, EMS Tracts in Mathematics 24,
European Mathematical Society (EMS), Z\"urich, 2015.

\bibitem{tru}
N.~Trudinger, 
On imbeddings into Orlicz spaces and some 
applications. 
J. Math.\ Mech. 17(1967), 473-483.

\bibitem{yy02}
D. Yang and W. Yuan,
Relations among Besov-type spaces, Triebel-Lizorkin-type spaces and generalized
Carleson measure spaces,  Appl. Anal.
92 (2013), 549-561.

\bibitem{YHMSY}
W.~Yuan, D. D.~Haroske, S. D.~Moura, L.~Skrzypczak and D.~Yang,
Limiting embeddings in smoothness Morrey spaces, continuity envelopes and applications, 
{J. Approx. Theory}, 192 (2015) 306-335.

\bibitem{HSYY}
W. Yuan, D.D.~Haroske, L.~Skrzypczak and D. Yang,
Embedding Properties of Besov-Type Spaces, Applicable Anal., 94(2015), 318-340. 

\bibitem{ysy} 
W.~Yuan, W.~Sickel, and D.~Yang,
Morrey and Campanato Meet Besov, Lizorkin and Triebel, 
Lecture Notes in Mathematics 2005, Springer-Verlag, Berlin, 2010, xi+281 pp.

\end{thebibliography}
\end{document}